\documentclass[11pt,reqno]{amsart}

\usepackage[shortlabels]{enumitem}

\usepackage{color}
\usepackage{amsmath, amsthm, amssymb}
\usepackage{amsfonts}

\usepackage[ansinew]{inputenc}
\usepackage[dvips]{epsfig}
\usepackage{graphicx}
\usepackage[english]{babel}
\usepackage{hyperref}
\theoremstyle{plain}
\newtheorem{theorem}{Theorem}[section]
\newtheorem{corollary}[theorem]{Corollary}
\newtheorem{lemma}[theorem]{Lemma}
\newtheorem{proposition}[theorem]{Proposition}

\theoremstyle{definition}
\newtheorem{definition}[theorem]{Definition}

\theoremstyle{remark}
\newtheorem{remark}[theorem]{Remark}

\numberwithin{equation}{section}

\newcommand{\average}{{\mathchoice {\kern1ex\vcenter{\hrule height.4pt
				width 6pt depth0pt} \kern-9.7pt} {\kern1ex\vcenter{\hrule
				height.4pt width 4.3pt depth0pt} \kern-7pt} {} {} }}

\def\R{\mathbb{R}}
\def\S{\mathbb{S}}

\def\N{\mathbb{N}}

\newcommand{\loc}{{\mathrm{loc}}}

\begin{document}
	
	\title[Higher order parabolic boundary Harnack inequality]{Higher order parabolic boundary Harnack inequality in $C^1$ and $C^{k,\alpha}$ domains}
	
	\author{Teo Kukuljan}
	\address{Universitat de Barcelona,
		Departament de Matematiques i Inform\`atica, Gran Via de les Corts Catalanes 585,
		08007 Barcelona, Spain.}
	\email{tkukuljan@ub.edu}
	
	\thanks{The author has received funding from the European Research Council (ERC) under the Grant Agreement No 801867, and from the Swiss National Science Foundation project 200021\_178795. 
	The paper was written under supervision of Xavier Ros-Oton, we thank him for proposing many ideas and approaches.}
	\keywords{Parabolic equations, boundary Harnack, free boundary, higher regularity.}
	\subjclass[2010]{35K10, 35R35.}
	
	\begin{abstract}
		We study the boundary behaviour of solutions to second order parabolic linear equations in moving domains. Our main result is a higher order boundary Harnack inequality  in $C^1$ and $C^{k,\alpha}$ domains, providing that the quotient of two solutions vanishing on the boundary of the domain is as smooth as the boundary. 
		
		As a consequence of our result, we provide a new proof of higher order regularity of the free boundary in the parabolic obstacle problem.
	\end{abstract}

	\maketitle
	
\section{Introduction}	

	We study the boundary behaviour of solutions of 
	\begin{equation}\label{theEquation}
	\left\lbrace\begin{array}{rcll}
		\partial_tu + Lu &=& f &\text{in }\Omega\cap Q_1\\
		u &=& 0 &\text{on }\partial\Omega\cap Q_1,
	\end{array}\right.
	\end{equation}
	where $\Omega\subset\R^{n+1}$ is an open set, $Q_1 = B_1\times (-1,1)$ and $L$ is a second order operator of the form 
	\begin{equation}\label{oparatorForm}
		L u(x,t) = - \sum_{i,j=1}^na_{ij}(x,t) \partial_{ij}u(x,t) + \sum_{i=1}^n b_i(x,t)\partial_i u(x,t).
	\end{equation}
	The matrix $A(x,t):=\left[a_{i,j}(x,t)\right]_{ij}$ is assumed to be uniformly elliptic and the vector $b(x,t) = \left[ b_i(x,t)\right] _{i}$ is bounded, that is, for every $(x,t)$ it holds 
	\begin{equation}\label{coefficientCondition}
		0<\lambda I\leq A(x,t)\leq\Lambda I,\quad\quad||b||_{L^\infty}\leq \Lambda.
	\end{equation}	
	
	In the elliptic version of the problem, classical boundary Schauder estimates imply that solutions vanishing on the boundary of a $C^{k,\alpha}$ domain are of class $C^{k,\alpha}$ up to the boundary. This implies that the quotient of the solution with the distance function to the boundary is $C^{k-1,\alpha}$. This has been refined by De Silva and Savin \cite{DS15}, who showed that replacing the distance function with another solution gives that the quotient is $C^{k,\alpha}$ as well, under positivity assumption on the solution in the denominator. This is a higher order boundary Harnack estimate. 
	
	Such result was extended  to the parabolic setting in \cite{BG16}, where Banerjee and Garofalo established a new parabolic higher order boundary Harnack estimate in $C^{k,\alpha}$ domains.
	
	The goal of this paper is twofold. On the one hand, we give a new proof of the result in \cite{BG16}, which is based on blow-up and contradiction arguments. On the other hand, we extend the result to the case of parabolic $C^1$ domains, which was not covered in \cite{BG16}. This is important when applying these estimates to obtain the higher regularity of free boundaries in parabolic obstacle problems, as we will see below.
	
	We work in domains $\Omega$ which are allowed to be moving in time, and have to satisfy some parabolic regularity properties. 
	\begin{definition}\label{assumptionsOnDomain}
		Let $\beta\geq 1$ and let $C^\beta_p$ be the parabolic H\"older space defined in Definition~\ref{parabolicHolderSpaces}. We say that $\Omega\subset \R^{n+1}$ is $C^\beta_p$ in $Q_1$, if there is a function $F\colon Q_1'\to\R$, with $||F||_{C^\beta_p(Q_1')}\leq C_0$ and $F(0,0) =0,$ $\nabla_{x'} F(0,0) = 0$ so that
		$$\Omega\cap Q_1 = \{(x',x_n,t);\text{ }x_n>F(x',t)\}.$$
	\end{definition}
	
	For a domain $\Omega$  as in the above definition we define the parabolic distance to the boundary as follows
	$$d(x,t) = \inf_{(z,s)\in\partial\Omega\cap Q_1}\left( |x-z|+|t-s|^{\frac{1}{2}}\right),\quad (x,t)\in \Omega\cap Q_1.$$
	
	Our first main result says that if we have two solutions of \eqref{theEquation} in a $C^1_p$ domain, with a bounded right-hand side, then their quotient is nearly as smooth as the boundary. This is a parabolic boundary Harnack inequality, and reads as follows.
	
	\begin{theorem}\label{regularityOfQuotient2}
		Let $\Omega$ be $C^1_p$ in $Q_1$, and let $L$ be as in \eqref{oparatorForm}--\eqref{coefficientCondition}, with $A\in C(\overline{\Omega})$ and $b\in L^\infty(\Omega).$ Let $u_i$ be two solutions of \eqref{equationForTwoSolutions}	
		with $f_i\in L^\infty(\Omega\cap Q_1).$ Assume that $|u_2|\geq c_2 d$ with $c_2>0$ and $||f_2||_{L^{\infty}(\Omega\cap Q_1)} + ||u_2||_{L^\infty(\Omega\cap Q_1)}\leq C_2.$
		
		Then for any $\varepsilon\in(0,1)$ we have
		$$\left|\left| \frac{u_1}{u_2}\right|\right| _{C^{1-\varepsilon}_p(\overline{\Omega}\cap Q_{1/2})}\leq C \left(||f_1||_{L^{\infty}_p(\Omega\cap Q_1)} + ||u_1||_{L^\infty(\Omega\cap Q_1)}\right),$$
		with $C$ depending only on $n,\varepsilon,c_2,C_2,\Omega$, the modulus of continuity of $A$ and ellipticity constants.
	\end{theorem}

	It is established through expanding one solution with respect to the other one, combined with boundary regularity estimates for solutions. With a similar, but a bit more involved approach, we can also prove the higher order analogue of the result - the higher order parabolic boundary Harnack inequality. 
	
	\begin{theorem}\label{regularityOfQuotient}
		Let $\beta> 1$, $\beta\not\in\N$. Let $\Omega\subset\R^{n+1}$ be $C^\beta_p$ in $Q_1$ according to Definition~\ref{assumptionsOnDomain}. Let $L$ be an operator of the form \eqref{oparatorForm} satisfying conditions \eqref{coefficientCondition} with the coefficients $A,b\in C^{\beta-1}_p(\overline{\Omega})$. For $i=1,2$ let $u_i$ be a solution to 
		\begin{equation}\label{equationForTwoSolutions}
			\left\lbrace\begin{array}{rcl l}
			\partial_t u_i + Lu_i & =& f_i &\text{in } \Omega\cap Q_1\\
			u_i& =& 0 &\text{on } \partial\Omega\cap Q_1,\\ 
			\end{array} \right.
		\end{equation}
		with $f_i\in C^{\beta-1}_p(\Omega\cap Q_1)$. Assume that $|u_2|\geq c_2 d$ with $c_2>0$ and $||f_2||_{C^{\beta-1}_p(\Omega\cap Q_1)} + ||u_2||_{L^\infty(\Omega\cap Q_1)}\leq C_2.$
		
		Then we have
		$$\left|\left| \frac{u_1}{u_2}\right|\right| _{C^\beta_p(\overline{\Omega}\cap Q_{1/2})}\leq C \left(||f_1||_{C^{\beta-1}_p(\Omega\cap Q_1)} + ||u_1||_{L^\infty(\Omega\cap Q_1)}\right),$$
		with $C$ depending only on $n,\beta,c_2,C_2,\Omega$, ellipticity constants, $||A||_{C^{\beta-1}_p(\overline{ \Omega})}$ and $||b||_{C^{\beta-1}_p(\overline{ \Omega})}$.
	\end{theorem}
	
	Theorem \ref{regularityOfQuotient} is already known under a bit stricter assumptions, see \cite{BG16}. We provide a new, different proof, which relaxes the assumption on regularity of the domain from  $C^{1,\alpha}$ in space and time, to $C^{1,\alpha}_p$.

	One motivation for studying such equations comes from the parabolic obstacle problem, where we look for a function $v\colon Q_1 \to\R$ solving
	\begin{equation}\label{theParabolicObstacleProblem}
	\partial_t v - \Delta v = -f\chi_{\{v>0\}},\quad v\geq 0,
	\end{equation}
	for some function $f$ depending only on $x$,
	together with some boundary data on $\partial_p Q_1 = \partial Q_1 \backslash( B_1\times \{1\}).$ Of particular interest is the so called \textit{contact set} $K = \{v=0\},$ and its boundary $\partial K,$ called the \textit{free boundary.} If we denote $\Omega = \{v>0\}$, then in particular the solution $v$ solves
	$$\left\lbrace\begin{array}{rcll}
	\partial_t v-\Delta v &=& -f &\text{in }\Omega\cap Q_1\\
	v&=&0&\text{on }\partial\Omega\cap Q_1.
	\end{array}\right.$$

	Theorem \ref{regularityOfQuotient2} and Theorem \ref{regularityOfQuotient} give a simple proof of the higher regularity of the free boundary in the parabolic obstacle problem, without making use of a hodograph transform as in \cite[Theorem 3]{KN77}. 
	Analogous boundary Harnack inequalities give rise to the bootstrap argument, which yields the higher order regularity of the free boundaries in similar obstacle problems. Some examples are \cite{DS15}, where they establish it in the classical case, \cite{JN18} for the case when the operator is the fractional Laplacian and \cite{BGZ17} for the parabolic thin obstacle problem.
	
	\begin{corollary}\label{infiniteRegularityOfFreeBoundary}
		Let $v\colon Q_1\to \R$ solve the parabolic obstacle problem
		\eqref{theParabolicObstacleProblem}
		with $f\geq c_0>0$, $f\in C^{\theta}(B_1)$, for some $\theta>1$.
		Assume that $(0,0)\in\partial\{v>0\}$, that $\partial\{v>0\}$ is $C^{1}_p$ in $Q_1$ in the sense of Definition \ref{assumptionsOnDomain}. Then $\partial\{v>0\}$ is $C^{\theta+1}_p$ in $Q_{1/2}$.
		
		In particular, if $f\in C^\infty(B_1),$ then the free boundary is $C^\infty$ near the origin in space and time.
	\end{corollary}
	\begin{remark}
		In the classical one-phase Stefan problem the function $f$ in the right-hand side is constantly equal to $1$ and moreover one has $\partial_t v\geq 0$. Then the initial $C^{1,\alpha}$ regularity of the free boundary near regular free boundary points\footnote{Regular points are those at which the blow-up is a $1$-dimensional profile of the form $\frac{1}{2}(x\cdot e)_+^2$; see \cite{LM15} for more details.} 
		is established in \cite{Ca77}. For general right-hand side $f$ in \eqref{theParabolicObstacleProblem}, we refer to \cite[Theorem 1.7]{LM15}. There the provided initial regularity is indeed $C^{1}_p$ only.
	\end{remark}

	Let us stress, that one often does not have $C^{1,\alpha}$ initial regularity. Another example is the free boundary problem for fully non-linear parabolic equations, studied in \cite{FS15}. There they establish the $C^1$ regularity of the free boundary in both space and time. This is why it is important to have Theorem \ref{regularityOfQuotient2} in $C^1_p$ domains in order to establish the higher regularity of free boundaries.
	

	\subsection{Strategy of the proofs}
	The main ingredient for proving Theorem \ref{regularityOfQuotient2} and Theorem \ref{regularityOfQuotient} is establishing an expansion result of the form 
	$$|u_1(x,t)-p(x,t)u_2(x,t)|\leq  C\left(|x|^{\beta+1}+|t|^{\frac{\beta+1}{2}}\right),\quad (x,t)\in Q_1,$$
	for some polynomial $p$ of parabolic degree $\lfloor\beta\rfloor$. To get such an expansion, we perform a contradiction in combination with a blow-up argument. The contradiction is reached with Liouville theorem in half-space for the parabolic setting. We follow the ideas of an analogous result in the non-local elliptic setting \cite{AR20}. In particular our proof is different from the one of De Silva-Savin in \cite{DS15}.
	
	\subsection{Notation}
	For a real number, we denote $\lfloor \cdot\rfloor $ its integer part and $\langle x\rangle = x - \lfloor x\rfloor$. 
	
	The ambient space is $\R^{n+1} = \R^n\times\R,$ where the first $n$ coordinates we denote with $x = (x_1,\ldots,x_n)$ and the last one with $t$. Sometimes we furthermore split $x=(x',x_n)$, for $x'=(x_1,\ldots,x_{n-1})$. 
	Accordingly we use the multi-index notation $\alpha\in\N_0^{n+1}$, $\alpha = (\alpha_1,\ldots,\alpha_n,\alpha_t)$, $|\alpha|=\alpha_1+\ldots+\alpha_n+\alpha_t$, and furthermore
	$$\partial^\alpha = \left(\frac{\partial}{\partial x_1}\right)^{\alpha_1}\circ\ldots\circ \left(\frac{\partial}{\partial x_n}\right)^{\alpha_n}\circ \left(\frac{\partial}{\partial t}\right)^{\alpha_t}.$$
	The gradient operator $\nabla$ and Laplace operator $\Delta$  are taken only in $x$ variables. When we want to use the full coordinate operators, we denote it with a subscript $\nabla_{(x,t)}$.
	
	For $\Omega\subset\R^{n+1}$ we denote the time slits with $\Omega_t = \{x\in\R^n;\text{ }(x,t)\in\Omega \}$ and the distance function to the boundary in space directions only with $d_t(x) = d_x(x,t) = \inf_{z\in\partial\Omega_t}|z-x|.$ Moreover its parabolic boundary $\partial_p\Omega$ is the set $\{(x,t)\in\partial\Omega;\text{ } \forall r>0: \text{ }B_r(x)\times(t-r,t)\not\subset\Omega \}.$
	
	We denote with $Q_r(x_0,t_0)$ the following parabolic cylinder of radius $r$ centred at $(x_0,t_0)$, 
	$$Q_r(x_0,t_0)= B_r(x_0)\times (t_0-r^2,t_0+r^2).$$ 
	When $(x_0,t_0) = (0,0)$, we denote it simply $Q_r$. Sometimes we also denote 
	$$Q_r^+ = Q_r\cap\{x_n>0\}.$$
	
	Finally, $C$ indicates an unspecified constant not depending on any of the relevant quantities, and whose value is allowed to change from line to line. We make use of sub-indices whenever we will want to underline the  dependencies of the	constant.

	\subsection{Organisation of the paper}
	In Section \ref{preliminaries}, we present the definitions of parabolic polynomials, parabolic derivatives and parabolic H\"older spaces, notion of viscosity solutions, and cite the results about them which we use. 
	In Section \ref{boundaryRegularityResults} we establish basic results regarding the behaviour of solutions of \eqref{theEquation} near the boundary.  Well equipped in Section \ref{sectionBoundaryHarnackC1pDomains} we prove the boundary Harnack estimate in $C^1_p$ domains and Theorem \ref{regularityOfQuotient2}. In the following Section \ref{higherOrdaryBoundaryShcauder...} we prove the higher order boundary Schauder estimates and Theorem \ref{regularityOfQuotient}. Finally, the proof of Corollary \ref{infiniteRegularityOfFreeBoundary} is presented in Section \ref{higherRegularityOfFree...}. At the end there is an appendix, where we establish the results regarding the interior regularity of solutions of \eqref{theEquation} needed for our specific setting. Moreover we there prove technical auxiliary results, to lighten the body of the paper.

\section{Preliminaries}	\label{preliminaries}
\subsection{Parabolic H\"older spaces}

	We follow the definitions of the parabolic H\"older spaces from \cite{Li96}. We start with defining parabolic derivatives and parabolic polynomial spaces.
	\begin{definition}
		For a multi-index $\alpha\in \N_0^{n+1}$, we denote $\alpha_x$ the first $n$ components and $\alpha_t$ the last one and define
		$$|\alpha|_p = |\alpha_x| + 2\alpha_t = \sum_{i=1}^n\alpha_i + 2\alpha_t.$$
		Furthermore, we define the parabolic derivatives with
		$$D^k_p = \{\partial^\alpha;\text{ }|\alpha|_p = k\}.$$
		Similarly, we define the parabolic polynomial spaces as follows:
		$$\textbf{P}_{k,p} = \left\{ \sum_{|\alpha|_p\leq k}c_\alpha (x,t)^\alpha;\text{ }c_\alpha \in\R\right\}.$$
		We say that $k$ is the parabolic degree of polynomial $p$, if $k$ is the least integer so that $p\in\textbf{P}_{k,p}.$
		For a polynomial $p = \sum_\alpha c_\alpha  (x,t)^\alpha$, we denote
		$$||p|| = \sum_\alpha|c_\alpha|.$$ 
	\end{definition}

	Next we define parabolic H\"older seminorms and spaces.

	\begin{definition}\label{parabolicHolderSpaces}
		Let $\Omega$ be an open subset of $\R^{n+1}$. For $\alpha\in(0,1]$ we define the parabolic H\"older seminorm of order $\alpha$ as follows
		$$\left[ u\right]_{C^\alpha_p(\Omega)}= \sup_{(x,t),(y,s)\in\Omega}\frac{|u(x,t)-u(y,s)|}{|x-y|^\alpha+|t-s|^{\frac{\alpha}{2}}},$$
		and
		$$\left[ u\right]_{C^{\alpha}_t(\Omega)} = \sup_{(x,t),(x,s)\in\Omega}\frac{|u(x,t)-u(x,s)|}{|t-s|^{\alpha}}.$$
		If $\alpha\in(1,2]$, we set
		\begin{align*}
			\left[ u\right]_{C^\alpha_p(\Omega)}&=  \left[ \nabla u\right]_{C^{\alpha-1}_p(\Omega)}+\left[ u\right]_{C^{\frac{\alpha}{2}}_t(\Omega)}.
		\end{align*}
		For bigger numbers $\alpha>2$, we set
		$$\left[ u\right]_{C^\alpha_p(\Omega)}  = \left[ \nabla u\right]_{C^{\alpha-1}_p(\Omega)}+\left[\partial_t u\right]_{C^{\alpha-2}_p(\Omega)}.$$
		When $\alpha\not\in\N$ we say that $u\in C^\alpha_p(\Omega)$, when $\left[ u\right]_{C^\alpha_p(\Omega)}<\infty$, and define 
		$$||u||_{C^\alpha_p(\Omega)} = \sum_{k\leq \lfloor\alpha\rfloor} ||D^k_p u||_{L^\infty(\Omega)} + 	\left[ u\right]_{C^\alpha_p(\Omega)}.$$
		
		If $\alpha\in\N$, we say that $u\in C^\alpha_p(\Omega)$, if there exists a modulus of continuity $\omega\colon \left[ 0, \infty\right) \to \left[ 0,\infty\right) $ -- a continuous, increasing function with $\omega(0)=0$ -- so that for all $(x,t),(y,s)\in\Omega$
		$$|D^\alpha_p u(x,t) - D^\alpha_p u(y,s)|\leq \omega(|x-y|+|t-s|^{\frac{1}{2}}),$$
		and
		$$|D^{\alpha-1}_p u(x,t) - D^{\alpha-1}_p u(x,s)|\leq |t-s|^{\frac{1}{2}}\omega(|t-s|^{\frac{1}{2}}).$$
		We set
		$$||u||_{C^\alpha_p(\Omega)} = \sum_{k\leq \alpha} ||D^k_p u||_{L^\infty(\Omega)} .$$
		Moreover, we say $u\in C^{\alpha-1,1}_p(\Omega)$ if 
		$\left[ u\right]_{C^\alpha_p(\Omega)}<\infty$
		and we set
		$$||u||_{C^{\alpha-1,1}_p(\Omega)} = \sum_{k\leq \alpha-1} ||D^k_p u||_{L^\infty(\Omega)} +\left[ u\right]_{C^\alpha_p(\Omega)}  .$$
	\end{definition} 
    
    \begin{remark}
    	Breaking down the definition of the parabolic H\"older seminorms, we get that
    	$$\left[ u\right]_{C^\alpha_p(\Omega)}  = \left[ D^{\lfloor\alpha\rfloor}_p u\right]_{C^{\langle\alpha\rangle}_p(\Omega)}+\left[D^{\lfloor\alpha\rfloor-1}_p u\right]_{C^{\frac{\langle\alpha\rangle+1}{2}}_t(\Omega)}.$$
    	
    	Notice that the seminorms are compatible with parabolic scaling. If $u_r(x,t) = u(rx, r^2t)$, then 
    	$$\left[ u_r\right]_{C^\alpha_p(Q_1)} = r^\alpha \left[ u\right]_{C^\alpha_p(Q_r)}.$$
    	
    	
    	We want to remark as well, that for every power $\beta>0$, we have
    	$$ C^{-1}\left(|x|+|t|^{\frac{1}{2}}\right)^\beta\leq |x|^\beta + |t|^{\frac{\beta}{2}}\leq C\left(|x|+|t|^{\frac{1}{2}}\right)^\beta.$$
    \end{remark}    

\subsection{Viscosity solutions} 
	For the notion of solutions of the equation \eqref{theEquation}
	we use the viscosity solutions.
	
	\begin{definition}
		Let $\Omega\subset\R^{n+1}$ be an open set. We say that $u\in C(\overline{\Omega})$ is a viscosity sub-solution of 
		$$\partial_t u + Lu = f$$	
	 	in $\Omega$, if for any function $\phi\in C^2_p(\Omega)$ touching $u$ from below at some point $(x_0,t_0)\in\Omega$ it holds
		$$\partial_t\phi(x_0,t_0) + L\phi(x_0,t_0)\leq f(x_0,t_0).$$
		Analogously we define viscosity super-solutions; the opposite inequality has to hold for $C^2_p(\Omega)$ functions touching $u$ from above.
		
		We say that $u\in C(\overline{\Omega})$ is a viscosity solution of \eqref{theEquation} if it is both viscosity sub-solution and viscosity super-solution.
	\end{definition}
	
	One of the main reasons for using viscosity solutions is, because they are well behaved under uniform convergence. We often use the stability result \cite[Theorem 1.1]{Ba05}. The existence and uniqueness result is established in \cite[Theorem 5.15]{Li96}, while the interior regularity results are provided in \cite[Theorem 2]{PS19}, when the right-hand side is bounded only, and in \cite[Theorem 5.9]{Li96}, when the right-hand side is H\"older continuous.

\section{Basic boundary regularity results}\label{boundaryRegularityResults}
	Since solutions of \eqref{theEquation} satisfy the comparison principle (see for example \cite[Lemma 2.1]{Li96}), a first key to get estimates for solutions near the boundary is constructing suitable barriers. 
	
	Let us start with providing the existence of a so called \textit{generalised distance} function. It is comparable to the parabolic distance to the boundary, but moreover satisfies some additional interior regularity properties, which are very important later on. The parabolic distance at a point $(x,t)\in\Omega$ is given as $$\operatorname{dist}_p((x,t),\partial\Omega) = \inf_{(z,s)\in\partial\Omega}\left(|x-z|+|t-s|^{\frac{1}{2}}\right).$$
	
	\begin{lemma}\label{generalisedDistance}
		Let $\beta\in \left[ 1,2 \right) $ and let $\Omega\subset \R^{n+1}$ be $C^\beta_p$ in $Q_1$ in the sense of Definition \ref{assumptionsOnDomain}. Then there exists a function $d\colon \Omega\cap Q_1\to\R$, satisfying
		$$C^{-1}\operatorname{dist}_p(\cdot, \partial\Omega)\leq d\leq C \operatorname{dist}_p(\cdot, \partial\Omega),\quad d\in C^\beta_p(Q_1)\cap C^2_p(\Omega\cap Q_1),\quad |\partial_t d|+|D^2 d|\leq Cd^{\beta-2},$$
		$$ \frac{2}{3}\leq|\nabla d|\leq C.$$
		Moreover, when $\beta = 1$, we have
		$$|\partial_t d|+|D^2 d|\leq Cd^{-1}\omega(d),$$
		where $\omega$ is the modolus of continuity  of $\nabla F$. The constant $C$ depends only on $\beta$, $||F||_{C^\beta_p(Q'_1)}$ and $n$.
	\end{lemma}
	
	\begin{proof}
		The construction is done in \cite[Section IV.5]{Li96}. The second part ($\beta = 1$) is established analogously as \cite[Theorem 1.3]{Li85}.
	\end{proof}
	
	With aid of the generalised distance we construct a barrier in $C^1_p$ and $C^{1,\alpha}_p$ domains. The barrier is a non-negative function, vanishing at the specific point of the boundary with positive right-hand side. When the boundary is $C^1_p$ only, the obtained barrier vanishes a bit slower than linearly.
	
	\begin{lemma}\label{barrier1}
		Let $\Omega\subset\R^{n+1}$ be $C^1_p$ in $Q_1$ in the sense of Definition \ref{assumptionsOnDomain}. Let $L$ be an operator of the form \eqref{oparatorForm} satisfying conditions \eqref{coefficientCondition}. Let $\varepsilon\in(0,1)$. Then for every boundary point $(z,s)\in\partial\Omega\cap Q_{1/2}$ there exists a function $\psi_{(z,s)}$ satisfying the following properties:
		$$\left\lbrace\begin{array}{rcl l}
		(\partial_t + L)\phi_{(z,s)} & \geq& d^{-1-\varepsilon}+1 &\text{in } \Omega\cap Q_1\cap \{d<r\}\\
		\phi_{(z,s)} & \geq& 1 &\text{in } \Omega\cap \partial_p\left( Q_1\cap \{d<r\}\right)\\
		\phi_{(z,s)}& \geq& 0 &\text{in } \Omega\cap Q_1\cap \{d<r\}\\ 
		\phi_{(z,s)}(z,s)& =& 0&\\
		\phi_{(z,s)}(z+\xi \nu_x,s)& \leq& Cd^{1-\varepsilon}(z+\xi \nu_x,s)&\text{for } \xi<r.
		\end{array} \right.$$
		The constants $C$ and $r>0$ depend only on $n,\varepsilon,\lambda,\Lambda$, $||F||_{C^{1}_p}$ and the modulus of continuity of $\nabla F$.
	\end{lemma}

	\begin{proof}
		Set first $\varphi = d^{1-\varepsilon}$. We compute
		\begin{align*}
			(\partial_t+L)\varphi& =    \varepsilon(1-\varepsilon) \nabla d^T A \nabla d\cdot d^{-1-\varepsilon}+ (1-\varepsilon)d^{-\varepsilon}(\partial_t+L)d \\
			&\geq   \varepsilon(1-\varepsilon) \frac{4\lambda}{9} d^{-1-\varepsilon} - C\Lambda(1-\varepsilon)d^{-1-\varepsilon}\omega(d) - C\Lambda (1-\varepsilon)d^{-\varepsilon},
		\end{align*}
		thanks to Lemma \ref{generalisedDistance}. If $(x,t)\in \{d<r\},$ for $r$ so that $\max(\omega(r),r)<\varepsilon\frac{\lambda}{9C\Lambda}$, we get
		$$(\partial_t+L)\varphi \geq  \varepsilon(1-\varepsilon) \frac{2\lambda}{9} d^{-1-\varepsilon}.$$
		The barrier can be taken as $\phi_{(z,s)}(x,t) = D\left(\varphi(x,t) + E(|x-z|^2+|t-s|^2)\right)$, for small enough $E>0$ and big enough $D$.
	\end{proof}
	
	In the case when the domain is $C^{1,\alpha}_p$, we can improve the barrier so that it grows linearly near the boundary.
	
	\begin{lemma}\label{barrier2}
		 Let $\beta \in (1,2)$ and let $\Omega\subset\R^{n+1}$ be $C^\beta_p$ in $Q_1$ in the sense of Definition \ref{assumptionsOnDomain}. Let $L$ be an operator of the form \eqref{oparatorForm} satisfying conditions \eqref{coefficientCondition}. Then for every boundary point $(z,s)\in\partial\Omega\cap Q_{1/2}$ there exists a function $\psi_{(z,s)}$ satisfying the following properties:
		 $$\left\lbrace\begin{array}{rcl l}
		 (\partial_t + L)\phi_{(z,s)} & \geq& d^{\beta-2}+1 &\text{in } \Omega\cap Q_1\cap \{d<r\}\\
		 \phi_{(z,s)} & \geq& 1 &\text{in } \Omega\cap \partial_p\left( Q_1\cap \{d<r\}\right)\\
		 \phi_{(z,s)}& \geq& 0 &\text{in } \Omega\cap Q_1\cap \{d<r\}\\ 
		 \phi_{(z,s)}(z,s)& =& 0&\\
		 \phi_{(z,s)}(z+\xi \nu_x,s)& \leq& Cd(z+\xi \nu_x,s)&\text{for } \xi<r.
		 \end{array} \right.$$
		 The constants $C$ and $r>0$ depend only on $n,\beta,\lambda,\Lambda$ and $||F||_{C^{\beta}_p}.$
	\end{lemma}

	\begin{proof}
		Set $\varphi = d - M d^\beta$, for some positive constant $M$ specified later. We compute 
		$$(\partial_t + L)\varphi = (\partial_t + L)d (1 - M\beta d^{\beta-1}) + M\beta(\beta-1) d^{\beta-2} \nabla d^T A \nabla d.$$
		Hence we can estimate
		\begin{align*}
			(\partial_t + L)\varphi \geq & -|(\partial_t + L)d| (1 + M\beta d^{\beta-1}) + M\beta(\beta-1)d^{\beta-2} \nabla d^T A \nabla d \\ 
			\geq& -Cd^{\beta-2}(1+2Md^{\beta-1}) + \frac{4M}{9}\beta(\beta-1)d^{\beta-2}
		\end{align*}
		Choosing $M$ so that $\frac{4M}{9}\beta(\beta-1)>\max\{3C, 1\}$, and then $r$ so that $2Md^{\beta-1}< 1$, we get that for $x\in\{d<r\}\cap \{d<1\}$
		$$(\partial_t + L)\varphi \geq d^{\beta-2}\geq \frac{1}{2}(d^{\beta-2} + 1).$$
		We now set $\phi_{(z,s)}(x,t) = D\left(\varphi(x,t) + E((t-s)^2 + |z-x|^2)\right)$. Constant $E$ has to be small enough so that $D^{-1}(\partial_t + L)\phi_{(z,s)}\geq \frac{1}{4}(d^{\beta-2}+1)$, and then $D$ is chosen big enough, so that $(\partial_t + L)\phi_{(z,s)}\geq d^{\beta-2} + 1$ and that $\phi_{(z,s)}\geq 1$ on $\Omega\cap \partial_p \left(\{d<r\}\cap Q_1 \right)$.
		Since $\varphi(z+\xi \nu_x,s)\leq Cd(z+\xi\nu_x,s)$, provided that $r$ is small enough, the same inequality holds for $\phi_{(z,s)}$ as well.
	\end{proof}

	The existence of such barriers straight forward imply that solutions to parabolic equations grow near the boundary at most linearly (almost linearly if the boundary is $C^1_p$).
	
	\begin{corollary}
		Let $\beta\in [ 1,2).$ Let $\Omega\subset\R^{n+1}$ be $C^\beta_p$ in $Q_1$ in the sense of Definition \ref{assumptionsOnDomain} and let $L$ be an operator of the form \eqref{oparatorForm} satisfying conditions \eqref{coefficientCondition}. Suppose $u$ is a solution to 
		$$\left\lbrace\begin{array}{rcl l}
		\partial_t u + Lu & =& f &\text{in } \Omega\cap Q_1\\
		u& =& 0 &\text{on } \partial\Omega\cap Q_1.\\ 
		\end{array} \right.$$
		If $\beta = 1$,then for every $\varepsilon>0$ have 
		$$|u|\leq C \big (||fd||_{L^\infty(\Omega\cap Q_1)}+||u||_{L^\infty(\Omega\cap Q_1)}\big )d^{1-\varepsilon}\quad \text{in } \Omega\cap Q_{1/2},$$
		with $C$ depending only on $n,\varepsilon,\lambda,\Lambda,$ and $||F||_{C^{1}_p}.$
		 
		When $\beta>1$, we have 
		$$|u|\leq C \big (||fd^{2-\beta}||_{L^\infty(\Omega\cap Q_1)}+||u||_{L^\infty(\Omega\cap Q_1)}\big )d\quad \text{in } \Omega\cap Q_{1/2},$$
		with $C$ depending only on $n,\beta,\lambda,\Lambda,$ and $||F||_{C^{\beta}_p}.$
	\end{corollary}
	\begin{proof}
		Cases $\beta=1$ and $\beta>1$ are treated in the same way. Let us only prove the statement when $\beta>1.$
		
		Dividing $u$ with $||fd^{2-\beta}||_{L^\infty}+||u||_{L^\infty},$ we can assume that $|u|\leq 1$ and $|f|\leq d^{\beta-2}.$ Take $(x,t)\in \Omega\cap Q_{1/2}.$ If $d(x,t) < r/2$ for the constant $r$ and function $\phi_{(z,t)}$ from Lemma \ref{barrier2}, we apply the comparison principle (\cite[Lemma 2.1]{Li96})
		on $\phi_{(z,t)}\pm u$ and get that $|u(x,t)|\leq Cd(x,t).$ Taking $C$ bigger if necessary, the same holds also if $d_x(x,t) \geq  r/2$. Note that thanks to Definition \ref{assumptionsOnDomain} the function $d$ is also comparable to $d_x$. In case $\beta=1$ Lemma \ref{barrier1} together with comparison principle gives the desired result.
	\end{proof}

	Having bounds on the growth near the boundary, combined with interior estimates quickly give Lipschitz bounds up to the boundary (almost Lipschitz in $C^1_p$ domains).
	
	\begin{corollary}\label{lipschitzBounds}
		For $\beta\in [ 1,2)$ let $\Omega\subset\R^{n+1}$ be $C^\beta_p$ in $Q_1$ in the sense of Definition \ref{assumptionsOnDomain} and let $L$ be an operator of the form \eqref{oparatorForm} satisfying conditions \eqref{coefficientCondition}, with $A\in C^0(\overline{\Omega})$. Let $u$ be a solution to 
		$$\left\lbrace\begin{array}{rcl l}
		\partial_t u + Lu & =& f &\text{in } \Omega\cap Q_1\\
		u& =& 0 &\text{on } \partial\Omega\cap Q_1.\\ 
		\end{array} \right.$$
		
		If $\beta=1$ we have for every $\varepsilon>0$
		$$\left[ u\right]_{C^{1-\varepsilon}_p(\Omega\cap Q_{1/2})} \leq C \big (||fd^{2-\beta}||_{L^\infty(\Omega\cap Q_1)}+||u||_{L^\infty(\Omega\cap Q_1)}\big ),$$
		with $C$ depending only on $n,\varepsilon,\lambda,\Lambda,$ $||F||_{C^{1}_p},$ and the modulus of continuity of $A$.
		
		If $\beta>1$ we have
		$$\left[ u\right]_{C^{1}_p(\Omega\cap Q_{1/2})} \leq C \big (||fd^{2-\beta}||_{L^\infty(\Omega\cap Q_1)}+||u||_{L^\infty(\Omega\cap Q_1)}\big ),$$
		with $C$ depending only on $n,\beta,\lambda,\Lambda,$  $||F||_{C^{\beta}_p},$ and the modulus of continuity of $A$.
	\end{corollary}
	
	\begin{proof}
		Let first $\beta>1$.
		Take any $Q_r(x_0,t_0)\subset Q_{2r}(x_0,t_0)\subset \Omega$ so that $d(x_0,t_0)\leq C r$ with $C$ independent on $u,f,x_0,t_0$. \footnote{The constant $C$ can be taken as $\frac{\sqrt{S}-1}{S-1}$, where $S=\left[ F\right]_{C^{\frac{1}{2}}_t}$.} 
		Define $u_r(x,t) = u(x_0+2rx, t_0+4r^2t)$. It solves
		$$\partial_t u_r + Lu_r  = 4r^2f \quad \text{in }  Q_1.$$
		Hence by interior regularity results \cite[Theorem 2]{PS19}
		\begin{align*}
			\left[ u_r\right]_{C_p^{1}(Q_{1/2})} &\leq C \big (||4r^2f||_{L^\infty( Q_{2_r(x_0,t_0)})}+||u_r||_{L^\infty(Q_1)}\big )
			\\&\leq Cr\big (||fd^{2-\beta}||_{L^\infty(\Omega\cap Q_1)}+||u||_{L^\infty(\Omega\cap Q_1)}\big ), 
		\end{align*}
		where we applied the growth control on $u$. Translating it back to $u$, we get
		$$\left[ u\right]_{C_p^{1}(Q_{r})} \leq C \big (||fd^{2-\beta}||_{L^\infty(\Omega\cap Q_1)}+||u||_{L^\infty(\Omega\cap Q_1)}\big ). $$
		Due to Lemma \ref{fullRegularityFromCylinders} the claim is proven. 
		
		When $\beta =1$, the prove is analogous. The growth control of $u$ is $|u|\leq Cd^{1-\varepsilon}$, which assures that we can bound $\left[ u\right]_{C_p^{1-\varepsilon}(Q_{r})}$ instead.
	\end{proof} 

	Having boundary regularity estimates enables us to prove the following Liouville-type theorem in a half-space. Recall that $Q_R^+ = Q_R\cap\{x_n>0\}.$
	
	\begin{proposition}\label{Liouville}
		Assume that $v$ solves 
		$$\left\lbrace\begin{array}{rcl l}
		\partial_t v - \operatorname{tr}(AD^2v) & =& P &\text{on } \{x_n>0\}\\
		v&=&0&\text{in } \{x_n=0\}\\
		||v||_{L^\infty(Q_R^+ )}& \leq& C_0R^\gamma &\forall R>1,\\ 
		\end{array} \right.$$
		for some constant, uniformly elliptic matrix $A$, $\gamma>0,$ $\gamma\not\in\N$, and some polynomial $P\in\textbf{P}_{\lfloor\gamma-2\rfloor,p}$  (if $\gamma<2$, then $P=0$). Then $v = Q x_n$ for some $Q\in\textbf{P}_{\lfloor\gamma-1\rfloor,p}.$
	\end{proposition}  

	\begin{proof}
		Again define
		$$v_R(x,t) := R^{-\gamma} v(Rx,R^2t).$$
		Then $||v_R||_{L^\infty(Q_{1})}\leq C_0$, $(\partial_t - \operatorname{tr}(AD^2))v_R(x,t) = R^{2-\gamma} P(Rx,R^2t),$ and $v_R=0$ on $\{x_n=0\}.$ Hence applying Corollary \ref{lipschitzBounds}, we get
		$$\left[ v_R\right] _{C_p^{1}(Q_{1/2}^+)}\leq C(||P|| + C_0),$$
		which translates to
		$$\left[ v\right] _{C_p^{1}(Q_{R/2}^+)}\leq CR^{\gamma-1}(||P|| + C_0),$$
		with $C$ independent of $R$ (same as in the previous proposition). Now we take arbitrary $h\in\R^n$ with $h_n=0$ and $\tau\in\R$, and define 
		$$w_1(x,t):= \frac{1}{|h|+|\tau|^{\frac{1}{2}}}(v(x+h,t+\tau)-v(x,t)),$$
		so that $||w_1||_{L^\infty(Q_R^+)}\leq \left[ v\right] _{C_p^{1}(Q_{R}^+)}\leq CR^{\gamma-1}$, as well as 
		$$\left\lbrace\begin{array}{rcl l}
		\partial_t w_1 - \operatorname{tr}(AD^2w_1) & =& P_1 &\text{in } \{x_n>0\}\\
		w_1&=&0&\text{in } x_n=0,\\
		\end{array} \right.$$
		for some polynomial $P_1\in \textbf{P}_{\lfloor\gamma-3\rfloor,p}$. After $K=\lceil \gamma\rceil$ same steps we conclude that $w_K$ is constant which implies that 
		$$v(x,t)=\sum_{|\alpha|\leq \gamma} v_{\alpha}(x_n) (x',t)^\alpha,$$
		for some functions $v_\alpha$ depending only on $x_n$.
		Choose now some maximal multi-index $\alpha$. Then $\partial^\alpha v = c_\alpha v_\alpha $ satisfies 
		$$\left\lbrace\begin{array}{rcl l}
		- a_{nn}\frac{d^2}{dx_n^2} v_\alpha  & =& \partial^\alpha P &\text{in } \{x_n>0\}\\
		v_\alpha&=&0&\text{in } \{x_n=0\},\\
		\end{array} \right.$$
		and hence $v_\alpha$ is a polynomial. Continuing in a similar manner and treating higher order coefficients as right-hand side, we conclude that $v$ is a polynomial itself. Noticing that $v$ satisfies suitable growth control and vanishes on $\{x_n=0\}$ we conclude the wanted result.
	\end{proof}

\section{Boundary Harnack estimate in $C^1_p$ domains}\label{sectionBoundaryHarnackC1pDomains}

	In this section we prove Theorem \ref{regularityOfQuotient2}.
	To establish it we first need to show the following expansion result, saying that solutions of \eqref{equationForTwoSolutions} can be well approximated near the boundary by another non-trivial solution.

	\begin{proposition}\label{HarnackExpansion0}
	Let $\Omega\subset\R^{n+1}$ be $C^1_p$ in $Q_1$ in the sense of Definition \ref{assumptionsOnDomain}, with $||F||_{C^{1}_p(Q_1')}\leq 1$.  Let $L$ be an operator of the form \eqref{oparatorForm} satisfying conditions \eqref{coefficientCondition}, with $A\in C^{0}(\Omega)$. For $i=1,2$ let $u_i$ be a solution to 
	$$\left\lbrace\begin{array}{rcl l}
	\partial_t u_i + Lu_i & =& f_i &\text{in } \Omega\cap Q_1\\
	u_i& =& 0 &\text{on } \partial\Omega\cap Q_1,\\ 
	\end{array} \right.$$
	with $f_i\in L^\infty(\Omega\cap Q_1)$. Assume that $|u_2|\geq c_0 d$ with $c_0>0$, $||u_i||_{L^\infty(\Omega)}\leq 1$ and $||f_i||_{L^\infty(\Omega)}\leq 1$. Let $\varepsilon\in(0,1)$. Then for every $(z,s)\in\partial\Omega\cap Q_{1/2}$ there exists a constant $c_{(z,s)}\in \R$, so that 
	$$|u_1(x,t)-c_{(z,s)} u_2(x,t)|\leq C\big (|x-z|^{2-\varepsilon} + |t-s|^{\frac{2-\varepsilon}{2}}\big ).$$
	The constant $C$ depends only on $n,\varepsilon,c_0$, the modulus of continuity of $A$ and ellipticity constants.
	
	Moreover for every $(x_0,t_0)\in \Omega\cap Q_{1/2}$, such that $d_t(x_0) = |x_0-z| = c_\Omega r,$ we have 
	\begin{equation}\label{interRegForExpansion0}
	\left[ u_1 - c_{(z,t_0)} u_2 \right] _{C^{2-\varepsilon}_p(Q_r(x_0,t_0))}\leq C.
	\end{equation}
\end{proposition}
\begin{proof}
	With the same reasoning as in Proposition \ref{shauderExpansion} we can assume that $(z,s)=(0,0)$. We prove the claim by contradiction. Suppose for $i=1,2$ and $k\in\N$, there exist $\Omega_k$ which are $C^1_p$ in the sense of Definition \ref{assumptionsOnDomain} with $(0,0)\in\partial\Omega_k$ and $||F_k||_{C^1_p(Q_1')}\leq 1$,  functions $u_{i,k}$, $f_{i,k}$ with $||u_{i,k}||_{L^\infty}\leq 1$, $u_{2,k}\geq c_0 d$, $||f||_{L^\infty}\leq1$ and operators $L_k =-\operatorname{tr}(A_kD^2) + b_k\nabla $, with the modulus of continuity of $A_k$ independent of $k$, so that
	$$\left\lbrace\begin{array}{rcl l}
	\partial_t u_{i,k} + L_ku_{i,k} & =& f_{i,k} &\text{in } \Omega_k\cap Q_1\\
	u_{i,k}& =& 0 &\text{on } \partial\Omega_k\cap Q_1, 
	\end{array} \right.$$
	but 
	$$\sup_k\sup_{r>0} r^{-2+\varepsilon}||u_{1,k}-c_ku_{2,k} ||_{L^\infty(Q_r)}  = \infty$$
	for any choice of constants $c_k$.
	We extend functions $u_{i,k}$ with $0$ outside of $\Omega_k$ to functions defined on whole $Q_1$. We define 
	$$c_{k,r} = \frac{\int_{Q_r} u_{1,k}u_{2,k}}{\int_{Q_r}u_{2,k}^2},$$
	so that $\int_{Q_r}(u_{1,k}-c_{k,r}u_{2,k})u_{2,k}=0$. Furthermore we define
	$$\theta(r) = \sup_k\sup_{\rho>r} \rho^{-2+\varepsilon}||u_{1,k}-c_{k,r}u_{2,k}||_{L^\infty(Q_\rho)}.$$
	Thanks to Lemma \ref{copied4.5} $\theta(r)\to\infty$ as $r\downarrow 0.$ Choose a sequence $r_m,k_m$, with $r_m\leq \frac{1}{m}$, so that 
	$$\frac{1}{r_m^{2-\varepsilon}\theta(r_m)}||u_{1,km}-c_{k_m,r_m}u_{2,k_m}||_{L^\infty(Q_{r_m})}\geq \frac{1}{2}.$$
	We define the blow-up sequence 
	$$v_m(x,t) := \frac{1}{r_m^{2-\varepsilon}\theta(r_m)}\left(u_{1,k_m}-c_{k_m,r_m}u_{2,k_m}\right)(r_mx,r_m^2t).$$
	Due to the choice of $r_m,k_m$ and $c_{k_m,r_m},$ we have $1/2\leq ||v_m||_{L^\infty(Q_1)}\leq 1$ and $\int_{Q_1} v_m u_{2,k_m} = 0$.
	
	We now turn our attention to constants $c_{k,\rho}$. Since $u_{2,k}>c_0d_k$, we can estimate
	\begin{align*}
	\rho|c_{k,\rho}-c_{k,2\rho}|\leq & C||c_{k,\rho}u_{2,k}-c_{k,2\rho}u_{2,k}||_{L^\infty(Q_\rho\cap\{d_k>\rho/2\})}\\
	\leq & C||u_{1,k}-c_{k,\rho}u_{2,k}||_{L^\infty(Q_\rho)} + C||u_{1,k}-c_{k,2\rho}u_{2,k}||_{L^\infty(Q_{2\rho})}\\
	\leq & C\theta(\rho) \rho^{2-\varepsilon} + C\theta(2\rho)(2\rho)^{2-\varepsilon}\leq C\theta(\rho)\rho^{2-\varepsilon},
	\end{align*}
	so that it holds
	$$|c_{k,\rho}-c_{k,2\rho}|\leq C\theta(\rho)\rho^{1-\varepsilon}.$$
	Iterating the inequality above we get for any $j\in\N$
	\begin{align*}
	|c_{k,\rho}-c_{k,2^j\rho}|\leq  \sum_{i=0}^{j-1}   |c_{k,2^i\rho}-c_{k,2^{i+1}\rho}|      \leq & C\sum_{i=0}^{j-1}\theta(2^i\rho)(2^i\rho)^{1-\varepsilon}\\
	\leq & C\theta(\rho) \rho^{1-\varepsilon}\sum_{i=0}^{j-1} \frac{\theta(2^i\rho)}{\theta(\rho)} 2^{i(1-\varepsilon)}\\
	\leq & C\theta(\rho)(2^j\rho)^{1-\varepsilon}.
	\end{align*}
	It follows that for any $R>1$, we have
	$$|c_{k,\rho}-c_{k,R\rho}|\leq C\theta(\rho)(R\rho)^{1-\varepsilon},$$
	and therefore
	$$||c_{k,\rho}u_{2,k} - c_{k,R\rho}u_{2,k}  ||_{L^\infty(Q_{R\rho})}\leq C\theta(\rho)(R\rho)^{2-\varepsilon}.$$
	Hence
	\begin{align*}
	&||v_m||_{L^\infty(Q_R)} =  \frac{1}{r_m^{2-\varepsilon}\theta(r_m)}||u_{1,k_m}-c_{k_m,r_m}u_{2,k_m}||_{L^\infty(Q_{Rr_m})}\\
	&\leq \frac{1}{r_m^{2-\varepsilon}\theta(r_m)}\left(||u_{1,k_m}-c_{k_m,Rr_m}u_{2,k_m}||_{L^\infty(Q_{Rr_m})} + ||c_{k_m,Rr_m}u_{2,k_m}-c_{k_m,r_m}u_{2,k_m}||_{L^\infty(Q_{Rr_m})} \right)\\
	&\leq \frac{1}{r_m^{2-\varepsilon}\theta(r_m)}\left( \theta(Rr_m)(Rr_m)^{2-\varepsilon} + C\theta(r_m)(Rr_m)^{2-\varepsilon} \right)\leq CR^{2-\varepsilon}.
	\end{align*}
	Moreover for each $\rho>0$ we have 
	$$\frac{|c_{k,\rho} - c_{k,2^j\rho}|}{\theta(\rho)}\leq C\sum_{i=0}^{j} \frac{\theta(2^{j-i}\rho)}{\theta(\rho)}(2^{j-i}\rho)^{1-\varepsilon},$$
	and choosing $j\in \N$ such that $2^j\rho\in \left[ 1,2\right) $, we deduce 
	$$\frac{|c_{k,\rho} - c_{k,2^j\rho}|}{\theta(\rho)}\leq C\sum_{i=0}^{j} \frac{\theta(2^{-i}\rho)}{\theta(\rho)}(2^{-i})^{1-\varepsilon}\longrightarrow 0 \quad \text{as }\rho\downarrow0.$$
	Hence, since $c_{k,\rho}$ is bounded for $\rho\in \left[ 1,2\right) $, we get
	$$\frac{|c_{k,\rho}|}{\theta(\rho)}\longrightarrow 0\quad \text{as }\rho\downarrow0,$$ uniformly in $k$.
	
	We compute 
	$$(\partial_t + \tilde{L}_{k_m})v_m(x,t) = \frac{r_m^\varepsilon}{\theta(r_m)}\left(f_{1,k_m}-c_{k_mr_m}f_{2,k_m}\right)(r_mx,r_m^2t),$$
	and hence
	$$|(\partial_t + \tilde{L}_{k_m})v_m(x,t)|\leq \frac{r_m^\varepsilon (1+c_{k_m,r_m})}{\theta(r_m)}.$$
	The right hand side converges to $0$ uniformly in compact sets in $\{x_n>0\}.$ Hence it is also bounded uniformly in $m$, and so by Lemma \ref{lipschitzBounds} we get uniform H\"older bounds $\left[ v_m\right]_{C^{1-\varepsilon}(Q_R)}\leq C_R$. Hence passing to a subsequence $v_m$ converges to $v$ locally uniformly in $\R^{n+1}$, that by \cite[Theorem 1.1]{Ba05} satisfies
	$$\left\lbrace\begin{array}{rcll}
	(\partial_t+L)v &=& 0&\text{in }\{x_n>0\}\\
	v&=&0&\text{on }\{x_n\leq0\},\\
	||v||_{L^\infty(Q_R)}&\leq &CR^{2-\varepsilon}&\text{for all } R\geq 1\\
	||v||_{L^\infty(Q_1)}&\geq& \frac{1}{2}&\\
	\int_{Q_1}v \cdot(x_n)_+ &=& 0,& 
	\end{array}\right.$$
	where $L$ is some constant coefficient $2$-homogeneous uniformly elliptic operator. Hence Theorem \ref{Liouville} says that $v = c (x_n)_+$ for some $c\in\R$. Therefore the last two properties of $v$ contradict each other.
	
	To prove \eqref{interRegForExpansion0}, we apply interior estimates \cite[Theorem 2]{PS19} on rescaled function $v_r(x,t) = (u_1-c_{(z,t_0)}v_2)(x_0+2rx,t_0+4r^2t)$, and take into account the above proven growth control and the uniform boundedness of $|c_{(z,t_0)}|$, see Lemma \ref{uniformBoundOnPolynomials}.
\end{proof}

	Combining the above result with estimates for $u_2$ yields Theorem \ref{regularityOfQuotient2}.
	
	\begin{proof}[Proof of Theorem  \ref{regularityOfQuotient2}]
		The proof goes along the same lines as \cite[Theorem 1.3]{AR20}.		
		
		Choose a cylinder $Q_r(x_0,t_0)$, so that $Q_{2_r}(x_0,t_0)\subset \Omega\cap Q_{3/4},$ and $d(x_0,t_0)\leq Cr$, with $C$ independent of $x_0,t_0$ and $r$. Let $z$ be the closest point to $x_0$ in $\partial\Omega_{t_0}$.
		Take arbitrary $(x,t)$ and $(y,s)$ in $Q_r(x_0,t_0)$, denote $c = c_{(z,t_0)}\in\R$ from the previous proposition, and compute
		$$\left|\frac{u_1}{u_2}(x,t) -\frac{u_1}{u_2}(y,s)\right| = 
		\left|\left(\frac{u_1}{u_2}(x,t) -c \right) -\left(\frac{u_1}{u_2}(y,s)-c \right)\right| \leq $$
		\begin{align}\label{termWithExpanisons1}
		\begin{split}
		\leq &  \left| \left( u_1 - c u_2 \right)(x,t) - \left( u_1 - c u_2 \right)(y,s) \right|\left|  u_2^{-1}(x,t) \right|+\\
		&+  \left| \left( u_1 - c u_2 \right)(y,s) \right|\left|  u_2^{-1}(x,t)- u_2^{-1}(y,s) \right|.\\
		\end{split}
		\end{align}
		By Proposition \ref{HarnackExpansion0} and Lemma \ref{fullRegularityFromCylinders} 
		we have that $u_1 - c u_2$ is $C^{2-\varepsilon}_p(\overline{\mathcal{C}}),$ for a cone $\mathcal{C} = \cup_{d_{t_0}(x_0) = c_\Omega r=|x_0-z|} Q_r(x_0,t_0)$, with $D^k_p(u_1 - p u_2)(z,t_0)=0$, for all $k< 2$. Hence by Lemma \ref{polynomialAproximation} we can estimate 
		$$|D^k_p (u_1 - p u_2) (x,t)|\leq C (|x-z|^{2-\varepsilon-k}+|t-t_0|^{\frac{2-\varepsilon-k}{2}}), \quad (x,t)\in\mathcal{C}, k< 2.$$
		
		We denote $v = u_1-cu_2$ and continue  estimating \eqref{termWithExpanisons1}:
		\begin{align*}
		| v(x,t)- v(y,s)|\leq& | v(x,t)- v(y,t)|+| v(y,t)-v(y,s)|\leq \\
		\leq& ||D^{1}_p v||_{L^\infty(Q_r(x_0,t_0))} |x-y| + \left[v\right]_{C^{2-\varepsilon}(Q_r(x_0,t_0))} |t-s|^{\frac{2-\varepsilon}{2}}\\
		\leq & C r^{1-\varepsilon} r^{\varepsilon}|x-y|^{1-\varepsilon} + C r^1 |t-s|^{\frac{1-\varepsilon}{2}}\\
		\leq & Cr(|x-y|^{1-\varepsilon}+|t-s|^{\frac{1-\varepsilon}{2}})
		\end{align*}
		Finally we also estimate
		\begin{align*}
			\left|  u_2^{-1}(x,t)- u_2^{-1}(y,s) \right|&\leq \left|  u_2^{-1}(x,t) u_2^{-1}(y,s) \right|\left|  u_2(x,t)- u_2(y,s) \right|\\
			&\leq Cr^{-2} \left[u_2\right]_{C^{1-\varepsilon}(\Omega\cap Q_{3/4})} (|x-y|^{1-\varepsilon}+|t-s|^{\frac{1-\varepsilon}{2}})\\
			&\leq C r^{-2+\varepsilon}(|x-y|^{1-2\varepsilon}+|t-s|^{\frac{1-2\varepsilon}{2}}),
		\end{align*}
		where we used Corollary \ref{lipschitzBounds} for $u_2$.
		Putting it all together we get
		$$\left|\frac{u_1}{u_2}(x,t) -\frac{u_1}{u_2}(y,s)\right|\leq C(|x-y|^{1-2\varepsilon}+|t-s|^{\frac{1-2\varepsilon}{2}}).$$
		
		Up to choosing a different $\varepsilon$ at the beginning, the claim is proven, thanks to Lemma \ref{fullRegularityFromCylinders}. 
	\end{proof}

\section{Higher order boundary Schauder and boundary Harnack estimates}\label{higherOrdaryBoundaryShcauder...}

	We proceed with finding a finer, higher order description of solutions near the boundary. The first goal is to get Schauder type expansions of solutions at boundary points - approximations with polynomials multiplied with the distance function. Note that the boundary Schauder estimates are already known; see for example \cite[Section IV.7]{Li96} or \cite[Section 10.3]{Kr96}. Still, the proof we present here is different and it is useful for the main result: the higher order parabolic boundary Harnack estimate. 
	
	\begin{proposition}\label{shauderExpansion}
		Let $\beta>1$, $\beta\not\in\N$. Let $\Omega\subset\R^{n+1}$ be $C^\beta_p$ in $Q_1$ in the sense of Definition \ref{assumptionsOnDomain}  with $||F||_{C^{\beta}_p(Q_1')}\leq 1$. Let $L$ be an operator of the form \eqref{oparatorForm} satisfying conditions \eqref{coefficientCondition}, with $A,b\in C^{\beta-2}_p(\overline{\Omega})$, if $\beta>2$ and $A\in C^0(\overline{\Omega})$, $b\in L^\infty(\Omega)$, if $\beta<2$. Let $u$ be a solution to 
		$$\left\lbrace\begin{array}{rcl l}
		\partial_t u + Lu & =& f &\text{in } \Omega\cap Q_1\\
		u& =& 0 &\text{on } \partial\Omega\cap Q_1,\\ 
		\end{array} \right.$$
		with $||f||_{C^{\beta-2}_p(\Omega)}\leq 1$ if $\beta>2$ and $|f|\leq d^{\beta-2}$ if $\beta<2.$ Assume that $||u||_{L^\infty(\Omega)}\leq 1$. Then for every $(z,s)\in\partial\Omega\cap Q_{1/2}$ there exists a polynomial $p_{(z,s)}\in \textbf{P}_{\lfloor\beta-1\rfloor,p}$, so that 
		$$|u(x,t)-p_{(z,s)}(x,t)d(x,t)|\leq C(|x-z|^\beta + |t-s|^{\frac{\beta}{2}}).$$
		The constant $C$ depends only on $n,\beta$, ellipticity constants and $||A||_{C^{\beta-2}_p(\overline{\Omega})}, ||b||_{C^{\beta-2}_p(\overline{\Omega})}$ if $\beta>2$, and the modulus of continuity of $A$, if $\beta<2$..
	\end{proposition}
	
	\begin{proof}
		Thanks to assumption on $\Omega$ (translations and scaling preserves the assumptions), we can assume that $(z,s) = (0,0).$ We prove the claim with contradiction argument. Assume that the claim is false. Then there exist $\Omega_k$ which are $C^\beta_p$ in $Q_1$ in the sense of Definition \ref{assumptionsOnDomain}, with $||F_k||_{C^\beta_p(Q_1')}\leq 1$ and $(0,0)\in\partial\Omega_k$, and $u_k,f_k,$ $L_k=\operatorname{tr}(A_kD^2) + b_k\cdot\nabla$, with uniformly bounded $||A_k||_{C^{\beta-2}_p(\overline{\Omega}_k)},||b_k||_{C^{\beta-2}_p(\overline{\Omega}_k)}$ if $\beta>2,$ and the modulus of continuity of $A_k$ independent of $k$, if $\beta<2$, satisfying $\partial_tu_k+L_ku_k = f_k$ in $\Omega_k\cap Q_1$. Moreover, $u_k=0$ on $\partial\Omega\cap Q_1,$ $|u_k|\leq 1$, $||f||_{C^{\beta-2}_p}\leq 1$ (or $|f_k|\leq d_k^{\beta-2}$ if $\beta<2$), $L_k$ is $(\lambda,\Lambda)$-uniformly elliptic, but for every polynomial $p_k\in\textbf{P}_{\lfloor\beta-1\rfloor,p}$ we have 
		$$\sup_{r>0}r^{-\beta}||u_k-p_kd_k||_{L^\infty(Q_r)}>k.$$
		We extend the functions $u_k$ and $d_k$ with zero in $Q_1\cap \Omega^c$ so that they are defined in the full cylinder $Q_1$, and denote them still $u_k,d_k$. Let $p_{k,\rho}d_k$ be the $L^2(Q_\rho)$ projection of $u_k$ to space $\textbf{P}_{\lfloor\beta-1\rfloor,p}d_k,$ so that we have
		$$\int_{Q_\rho} (u_k-p_{k,\rho}d_k)pd_k = 0,$$
		for every polynomial $p\in  \textbf{P}_{\lfloor\beta-1\rfloor,p}.$
		Then define the monotone quantity
		$$\theta(r) = \sup_k\sup_{\rho>r} \rho^{-\beta}||u_k-p_{k,\rho}d_k||_{L^\infty(Q_\rho)}.$$
		It follows from Lemma \ref{copied4.3} that $\lim_{r\downarrow0}\theta(r)=\infty.$
		Choose a sequence $k_m, r_m$ with $r_m\leq 1/m$, so that 
		$$\frac{1}{2}\theta(r_m)\leq r^{-\beta}||u_{k_m}-p_{k_m,r_m}d_{k_m}||_{L^\infty(Q_{r_m})},$$
		and define the blow-up sequence
		$$v_m(x,t) = \frac{1}{r_m^\beta\theta(r_m)}\big ( u_{k_m}(r_mx, r_m^2t) - p_{k_m,r_m}d_{k_m}(r_mx,r_m^2t) \big ).$$
		Notice that $||v_m||_{L^\infty(Q_1)}\geq 1/2$ and $\int_{Q_1} v_m p d_{k_m}(r_m\cdot,r_m^2\cdot) = 0$ for every polynomial $p\in\textbf{P}_{\lfloor\beta-1\rfloor,p}.$
		
		Let us now turn our attention to the polynomials $p_{k,\rho}$. We write 
		$$p_{k,\rho}(x,t) = \sum_{|\alpha|_p\leq \lfloor\beta-1\rfloor} p_{k,\rho}^{(\alpha)}(x,t)^\alpha,\quad p_{k,\rho}^{(\alpha)}\in\R.$$
		Using a rescaled version of \cite[Lemma A.10]{AR20} and that $d_k\geq c \rho$ in $Q_\rho\cap\{d_k>\rho/2\}$, we estimate for any $\alpha$ such that $|\alpha|_p\leq \lfloor\beta-1\rfloor$
		\begin{align*}
			\rho^{|\alpha|_p+1}|p_{k,\rho}^{(\alpha)}-p_{k,2\rho}^{(\alpha)}|\leq &C \rho ||p_{k,\rho}-p_{k,2\rho}||_{L^\infty(Q_\rho\cap\{d_k>\rho/2\})}\\
			\leq& C||p_{k,\rho}d_k-p_{k,2\rho}d_k||_{L^\infty(Q_\rho\cap\{d_k>\rho/2\})}\\
			\leq & C||u_k-p_{k,\rho}d_k||_{L^\infty(Q_\rho)} + C||u_k-p_{k,2\rho}d_k||_{L^\infty(Q_{2\rho})}\\
			\leq & C\theta(\rho) \rho^{\beta} + C\theta(2\rho)(2\rho)^\beta\leq C\theta(\rho)\rho^\beta,
		\end{align*}
		so that it holds
		$$|p_{k,\rho}^{(\alpha)}-p_{k,2\rho}^{(\alpha)}|\leq C\theta(\rho)\rho^{\beta-1-|\alpha|_p}.$$
		Iterating the inequality above we get for any $j\in\N$
		\begin{align*}
			|p_{k,\rho}^{(\alpha)}-p_{k,2^j\rho}^{(\alpha)}|\leq  \sum_{i=0}^{j-1}   |p_{k,2^i\rho}^{(\alpha)}-p_{k,2^{i+1}\rho}^{(\alpha)}|      \leq & C\sum_{i=0}^{j-1}\theta(2^i\rho)(2^i\rho)^{\beta-1-|\alpha|_p}\\
			\leq & C\theta(\rho) \rho^{\beta-1-|\alpha|_p}\sum_{i=0}^{j-1} \frac{\theta(2^i\rho)}{\theta(\rho)} 2^{i(\beta-1-|\alpha|_p)}\\
			\leq & C\theta(\rho)(2^j\rho)^{\beta-1-|\alpha|_p}.
		\end{align*}
		It follows that for any $R>1$, we have
		$$|p_{k,\rho}^{(\alpha)}-p_{k,R\rho}^{(\alpha)}|\leq C\theta(\rho)(R\rho)^{\beta-1-|\alpha|_p},$$
		and therefore
		$$||p_{k,\rho}d_k - p_{k,R\rho}d_k  ||_{L^\infty(Q_{R\rho})}\leq C\theta(\rho)(R\rho)^{\beta}.$$
		Hence
		\begin{align*}
			&||v_m||_{L^\infty(Q_R)} =  \frac{1}{r_m^{\beta}\theta(r_m)}||u_{k_m}-p_{k_m,r_m}d_{k_m}||_{L^\infty(Q_{Rr_m})}\\
			&\leq \frac{1}{r_m^{\beta}\theta(r_m)}\left(||u_{k_m}-p_{k_m,Rr_m}d_{k_m}||_{L^\infty(Q_{Rr_m})} + ||p_{k_m,Rr_m}d_{k_m}-p_{k_m,r_m}d_{k_m}||_{L^\infty(Q_{Rr_m})} \right)\\
			&\leq \frac{1}{r_m^{\beta}\theta(r_m)}\left( \theta(Rr_m)(Rr_m)^\beta + C\theta(r_m)(Rr_m)^\beta \right)\leq CR^\beta.
		\end{align*}
		Moreover for each $\rho>0$ we have 
		$$\frac{|p_{k,\rho}^{(\alpha)} - p_{k,2^j\rho}^{(\alpha)}|}{\theta(\rho)}\leq C\sum_{i=0}^{j} \frac{\theta(2^{j-i}\rho)}{\theta(\rho)}(2^{j-i}\rho)^{\beta-1-|\alpha|_p},$$
		and choosing $j\in \N$ such that $2^j\rho\in \left[ 1,2\right) $, we deduce 
		$$\frac{|p_{k,\rho}^{(\alpha)} - p_{k,2^j\rho}^{(\alpha)}|}{\theta(\rho)}\leq C\sum_{i=0}^{j} \frac{\theta(2^{-i}\rho)}{\theta(\rho)}(2^{-i})^{\beta-1-|\alpha|_p}\longrightarrow 0 \quad \text{as }\rho\downarrow0.$$
		Hence, since $p_{k,\rho}^{(\alpha)}$ is bounded for $\rho\in \left[ 1,2\right) $, we get
		$$\frac{|p_{k,\rho}^{(\alpha)}|}{\theta(\rho)}\longrightarrow 0\quad \text{as }\rho\downarrow0.$$
		
		We compute
		\begin{align*}
			(\partial_t+\tilde{L}_{k_m})v_m (x,t) &= \frac{r_m^{2-\beta}}{\theta(r_m)}\big ( f_{k_m} - (\partial_t - \tilde{L}_{k_m})(p_{k_m,r_m} d_{k_m}) \big )(r_mx, r_m^2t)\\
			&=\left\lbrace\begin{array}{cl}
			 \frac{r_m^{2-\beta}}{\theta(r_m)}o_{m,\beta}&\text{ if }\beta<2\\
			 P_m + \frac{1}{\theta(r_m)} o_{m,\beta}&\text{ if }\beta>2,
			\end{array}\right.
		\end{align*}
		where $P_m\in \textbf{P}_{\lfloor\beta-2\rfloor,p}$ is a suitable Taylor polynomial of quantities on the right-hand side.
		Note that $f_{k_m}$ and $ (\partial_t - \tilde{L}_{k_m})(p_{k_m,r_m} d_{k_m})$ are both $C^{\beta -2}_p$ functions, and they can be approximated with a polynomial up to order $|x|^{\beta-2}+|t|^{\frac{\beta-2}{2}}$ (see Lemma \ref{polynomialAproximation}), if $\beta>2$ and otherwise we bound the $L^\infty$ norm. The reminder we denote with $o_{m,\beta}$, so
		$$o_{m,\beta}(x,t) = \left\lbrace\begin{array}{cl}
		\big ( f_{k_m} - (\partial_t - \tilde{L}_{k_m})(p_{k_m,r_m} d_{k_m}) \big )(r_mx, r_m^2t)    &\text{ if }\beta<2\\
		\big ( f_{k_m}  - (\partial_t - \tilde{L}_{k_m})(p_{k_m,r_m} d_{k_m}) - P_m \big )(r_mx, r_m^2t)  &\text{ if }\beta>2,
		\end{array}\right.$$
		In the case $\beta<2$ we can estimate $\frac{r_m^{2-\beta}}{\theta(r_m)}|o_{m,\beta}|\leq C\frac{|p_{k_m,r_m}|}{\theta(r_m)} d_m^{\beta-2}(x,0)$, with $C$ independent of $m$, using the assumption on $f_m$,  estimates on $|\partial_t d|$ and $|D^2d_m|$ (see Lemma \ref{generalisedDistance}), while $\frac{p_{k_m,r_m}}{\theta(r_m)}$ converges to zero. In the case $\beta>2$, we have $|\theta(r_m)^{-1}o_{m,\beta}|\leq C\frac{||p_{k_m,r_m}||}{\theta(r_m)}(|x|^{\beta-2}+|t|^{\frac{\beta-2}{2}})$.
		Since $||v_m||_{L^\infty(Q_R)}\leq CR^\beta$ independently of $m$ and since $\frac{1}{\theta(r_m)} o_{m,\beta}$ are bounded uniformly in $m$, we get that $P_m$ are uniformly bounded as well (see the end of the proof of Proposition \ref{harnackExpansion}).
		Therefore in both cases the term with $o_{m,\beta}$ converges to zero in $L^\infty_\loc(\{x_n>0\}).$ Thanks to Corollary \ref{lipschitzBounds} we get uniform Lipschitz bounds for $v_m$ on every $Q_R$. Passing to a subsequence, the convergence result \cite[Theorem 1.1]{Ba05} assures the local uniform convergence of $v_m$ to some function $v$ defined in $\R^{n+1}\cap\{x_n>0\},$ satisfying
		$$\left\lbrace\begin{array}{rcll}
		(\partial_t+L)v &=& P&\text{in }x_n>0\\
		v&=&0&\text{in }x_n=0,\\
		||v||_{L^\infty(Q_R^+)}&\leq &CR^\beta&\\
		||v||_{L^\infty(Q_1^+)}&\geq& \frac{1}{2}&\\
		\int_{Q_1}v px_n &=& 0& \text{for every }p\in  \textbf{P}_{\lfloor\beta-1\rfloor,p},
		\end{array}\right.$$
		where $L$ is a constant coefficient, second order $(\lambda,\Lambda)$-elliptic operator and $P \in \textbf{P}_{\lfloor\beta-2\rfloor,p}$ if $\beta>2$ and $P=0$ if $\beta<2.$
		Hence by Liouville theorem (Proposition \ref{Liouville}) $v$ has to be equal to $qx_n$ for some polynomial $q$ in $ \textbf{P}_{\lfloor\beta-1\rfloor,p}$, which gives the contradiction.
	\end{proof}

	\begin{remark}
		When $\beta\in(1,2)$ the function $d$ is taken from Lemma \ref{generalisedDistance}. When $\beta>2$, we take the composition of the boundary flattening map ($\phi(x',x_n,t) = (x',x_n-F(x',t),t),$ with the $n$-th projection. Concretely, $d(x,t)=x_n-F(x',t)$, which is $C^\beta_p(Q_1).$ In both cases we get that $r^{-1}d(rx,r^2t)$ converges to $x_n$ locally uniformly.
		
		Additionally, the estimate $d(rx,r^2t)\leq C rd(x,rt)\leq C_1rd(x,0)$, works because the boundary is flat enough.
	\end{remark} 

	The obtained description of solutions near the boundary implies the regularity up to the boundary, as follows.
	
	\begin{corollary}\label{boundaryRegularity}
		Let $\beta>1$, $\beta\not\in\N$. Let $\Omega\subset\R^{n+1}$ be $C^\beta_p$ in $Q_1$ in the sense of Definition \ref{assumptionsOnDomain}. Let $L$ be an operator of the form \eqref{oparatorForm} satisfying conditions \eqref{coefficientCondition}, with $A,b\in C^{\beta-2}_p(\overline{\Omega})$, if $\beta>2$ and $A\in C^0(\overline{\Omega})$ if $\beta<2$. Let $u$ be a solution to 
		$$\left\lbrace\begin{array}{rcl l}
		\partial_t u + Lu & =& f &\text{in } \Omega\cap Q_1\\
		u& =& 0 &\text{on } \partial\Omega\cap Q_1,\\ 
		\end{array} \right.$$
		with $||f||_{C^{\beta-2}_p(\overline{\Omega}\cap Q_1)}\leq C_0$ if $\beta>2$ and $|f|\leq C_0 d^{\beta-2}$ if $\beta<2.$ 
		
		Then 
		$$\left| \left| u\right|\right|_{C^\beta_p(\Omega\cap Q_{1/2})}\leq CC_0.$$
		The constant $C$ depends only on $n,\beta,\Omega$, ellipticity constants and $||A||_{C^{\beta-2}_p(\overline{\Omega})}, ||b||_{C^{\beta-2}_p(\overline{\Omega})}$ if $\beta>2$, and the modulus of continuity of $A$, if $\beta<2$.
	\end{corollary}
	
	\begin{proof}
		Dividing the equation with $C_0$ if necessary, we may assume that $C_0=1$.
		Thanks to Lemma \ref{fullRegularityFromCylinders}, it is enough to prove
		$$\left[ u\right]_{C^\beta_p(Q_r(x_0,t_0))}\leq C,$$
		whenever $d_{t_0}(x_0) = |x_0-z|\leq Cr$, with $C$ independent of $x_0,t_0,r$, and $Q_{2r}(x_0,t_0)\subset\Omega\cap Q_{1/2}$. To prove that, take $p_{(z,t_0)}$ from Proposition \ref{shauderExpansion}, and define 
		$$u_r(x,t) := (u-p_{(z,t_0)}d)(x_0+2rx, t_0+(2r)^2t), \quad (x,t)\in Q_1.$$
		Then in the case $\beta>2$ use interior regularity estimates from Proposition \ref{interiorEstimates}, to get
		$$\left[ u_r\right]_{C^{\beta}_p(Q_{1/2})}\leq C\left(r^\beta||f +(\partial_t+L)(p_{(z,t_0)}d)||_{C^{\beta-2}_p(\Omega\cap Q_1)} + ||u - p_{(z,t_0)}d||_{L^\infty(Q_{2r}(x_0,t_0))}  \right)$$
		and hence 
		$$\left[ u-p_{(z,t_0)}d \right]_{C^{\beta}_p(Q_r(x_0,t_0))}\leq C,$$ thanks to Proposition \ref{shauderExpansion}, the fact that $p_{(z,t_0)}$ are uniformly bounded (see Lemma \ref{uniformBoundOnPolynomials}) and the fact that $d\in C^{\beta}_p(\overline{\Omega}\cap Q_1).$
		
		In case $\beta<2$, use \cite[Theorem 2]{PS19}
		to get
		$$\left[ u_r\right]_{C^{\beta}_p(Q_{1/2})}\leq C\left(r^2||f +(\partial_t+L)(p_{(z,t_0)}d)||_{L^\infty(Q_{2r}(x_0,t_0))} + ||u - p_{(z,t_0)}d||_{L^\infty(Q_{2r}(x_0,t_0))}  \right)$$
		and hence due to similar reasons as before
		$$\left[ u-p_{(z,t_0)}d \right]_{C^{\beta}_p(Q_r(x_0,t_0))}\leq C.$$
		
		This implies that 
		$$\left[ u\right]_{C^{\beta}_p(Q_r(x_0,t_0))}\leq C,$$
		thanks to regularity of $d$ and boundedness of $p_{(z,t_0)},$ as wanted.
	\end{proof}

	We can now establish expansions of one solution with respect to the other. This result is the key ingredient to prove the boundary Harnack estimate. It is the higher order version of Proposition \ref{HarnackExpansion0} Since the order exceeds $2$, the constant in the expansion is replaced with a polynomial, which causes some difficulties in the proof. 
	
	\begin{proposition}\label{harnackExpansion}
		Let $\beta>1$, $\beta\not\in\N$. Let $\Omega\subset\R^{n+1}$ be $C^\beta_p$ in $Q_1$ in the sense of Definition~\ref{assumptionsOnDomain}, with $||F||_{C^{\beta}_p(Q_1')}\leq 1$.  Let $L$ be an operator of the form \eqref{oparatorForm} satisfying conditions \eqref{coefficientCondition}, with $A,b\in C^{\beta-1}_p(\overline{\Omega})$. For $i=1,2$ let $u_i$ be a solution to 
		$$\left\lbrace\begin{array}{rcl l}
		\partial_t u_i + Lu_i & =& f_i &\text{in } \Omega\cap Q_1\\
		u_i& =& 0 &\text{on } \partial\Omega\cap Q_1,\\ 
		\end{array} \right.$$
		with $f_i\in C^{\beta-1}_p(\overline{\Omega}\cap Q_1)$. Assume that $|u_2|\geq c_0 d$ with $c_0>0$, $||u_i||_{L^\infty(\Omega)}\leq 1$ and $||f_i||_{C^{\beta-1}_p(\overline{\Omega})}\leq 1$. Then for every $(z,s)\in\partial\Omega\cap Q_{1/2}$ exists a polynomial $p_{(z,s)}\in \textbf{P}_{\lfloor\beta\rfloor,p}$, so that 
		$$|u_1(x,t)-p_{(z,s)}(x,t)u_2(x,t)|\leq C\big (|x-z|^{\beta+1} + |t-s|^{\frac{\beta+1}{2}}\big ).$$
		The constant $C$ depends only on $n,\beta,c_0$, $||A||_{C^{\beta-1}_p(\overline{\Omega})},||b||_{C^{\beta-1}_p(\overline{\Omega})}$ and ellipticity constants.
		
		Moreover for every $(x_0,t_0)\in \Omega\cap Q_{1/2}$, such that $d_t(x_0) = |x_0-z| = c_\Omega r,$ we have 
		\begin{equation}\label{interRegForExpansion}
			\left[ u_1 - p_{(z,t_0)} u_2 \right] _{C^{\beta+1}_p(Q_r(x_0,t_0))}\leq C.
		\end{equation}
	\end{proposition}
	
	\begin{proof}
		With the same reasoning as in Proposition \ref{shauderExpansion} we can assume $(z,s)=(0,0)$. 
		
		We write 
		$$p(x,t) = \sum_{|\alpha|_p\leq \beta} p_{\alpha}\cdot(x,t)^\alpha= p^{(0)}+\sum_{1\leq|\alpha|_p\leq \beta} p_{\alpha}(x,t)^\alpha= p^{(0)}+ p^{(1)}(x,t),$$
		for some constants $p_\alpha$.
		In view of Proposition \ref{shauderExpansion}
		$$u_2(x,t) = p_2(x,t)d(x,t) + v_2(x,t),\quad\quad p_2\in \textbf{P}_{\lfloor\beta-1\rfloor,p},\text{  }|v_2(x,t)|\leq C(|x|^{\beta}+|t|^{\frac{\beta}{2}}),$$
		and then the claim is equivalent to 
		$$\left|u_1(x,t) - p^{(0)}u_2(x,t) - p^{(1)}(x,t) p_2(x,t) d(x,t) - p^{(1)}(x,t)v_2(x,t)\right|\leq C\big (|x-z|^{\beta+1} + |t-s|^{\frac{\beta+1}{2}}\big ),$$
		which is furthermore equivalent to 
		$$\left|u_1(x,t) - \tilde{p}^{(0)}u_2(x,t) - \tilde{p}^{(1)}(x,t)  d(x,t)\right|\leq C\big (|x-z|^{\beta+1} + |t-s|^{\frac{\beta+1}{2}}\big ),$$
		for a suitable polynomial $\tilde{p}\in\textbf{P}_{\lfloor\beta\rfloor,p}$.
		
		We prove that by contradiction. So suppose for $i=1,2$, and $k\in\N$ there exist $\Omega_k,$ which are $C^\beta_p$ in $Q_1$ in the sense of Definition \ref{assumptionsOnDomain}, with $0\in\partial\Omega_k$ and $||F_k||_{C^\beta_p(Q_1')}\leq 1$, $u_{i,k},$ $f_{i,k}$ and $(\lambda,\Lambda)$-elliptic operator $L_k = -\operatorname{tr}(A_kD^2) + b_k\cdot\nabla$, with uniformly bounded norms $||A_k||_{C^{\beta-1}_p(\overline{\Omega}_k)},||b_k||_{C^{\beta-1}_p(\overline{\Omega}_k)}$, so that 
		$||u_{i,k}||_{L^\infty}\leq 1,$ $||f_{i,k}||_{C^{\beta-1}_p}\leq 1,$ $u_{2,k}\geq c_0d$ and 
		$$\left\lbrace\begin{array}{rcl l}
		\partial_t u_{i,k} + L_ku_{i,k} & =& f_{i,k} &\text{in } \Omega_k\cap Q_1\\
		u_{i,k}& =& 0 &\text{on } \partial\Omega_k\cap Q_1, 
		\end{array} \right.$$
		but 
		$$\sup_k\sup_{r>0} r^{-\beta-1}||u_{1,k}-p^{(0)}_ku_{2,k} - p^{(1)}_kd||_{L^\infty(Q_r)}  = \infty$$
		for every choice of  polynomials $p_k\in \textbf{P}_{\lfloor\beta\rfloor,p}$. We extend $u_{i,k}$ and $d_k$ to functions defined on whole $Q_1$ with zero outside of $\Omega_k$.
		Define $p_{k,r}$ as the $L^2(Q_r)$ projection of $u_{1,k}$ on the space $\{q^{(0)}u_{2,k} + q^{(1)}d_k; \text{ }q\in \textbf{P}_{\lfloor\beta\rfloor,p} \}$, so that 
		$$\int_{Q_r}(u_{1,k}-p_{k,r}^{(0)}u_{2,k}-p_{k,r}^{(1)}d_k)(q^{(0)}u_{2,k}+q^{(1)}d_k)=0$$
		for every polynomial $q\in\textbf{P}_{\lfloor\beta\rfloor,p}.$
		Furthermore, we define
		$$\theta(r) = \sup_k\sup_{\rho>r}\rho^{-\beta-1}||u_{1,k}-p_{k,\rho}^{(0)}u_{2,k} - p_{k,\rho}^{(1)}d_k||_{L^\infty(Q_\rho)}.$$
		Lemma \ref{copied4.5}, ensures that $\theta(r)\to \infty$ as $r\downarrow 0$. Pick a sequence $r_m, k_m$, with $r_m\leq 1/m$, so that 
		$$\frac{1}{r_m^{\beta+1}\theta(r_m)}||u_{1,k_m}-p_{r_m,k_m}^{(0)}u_{2,k_m} - p_{r_m,k_m}^{(1)}d_{k_m}||_{L^\infty(Q_{r_m})}\geq \frac{1}{2}.$$
		We define the blow-up sequence
		$$v_m(x,t) := \frac{1}{r_m^{\beta+1}\theta(r_m)}\left( u_{1,k_m}-p_{r_m,k_m}^{(0)}u_{2,k_m} - r_{r_m,k_m}^{(1)}d_{k_m}\right)(r_mx,r_m^2t).$$
		Note that $||v_m||_{L^\infty(Q_1)}\leq 1/2$, as well as $\int_{Q_1}v_m(q^{(0)}u_{2,k_m}+q^{(1)}d_{k_m}) = 0$ for every polynomial $q\in \textbf{P}_{\lfloor\beta\rfloor,p}.$
		With same arguments as in Proposition \ref{shauderExpansion}, we get $||v_m||_{L^\infty(Q_R)}\leq C R^{\beta+1}$, and $\frac{|p_{r,k}^{(\alpha)}|}{\theta(r)}\to 0$ uniformly in $k$ as $r\downarrow 0.$ 
		
		Now we turn our attention to 
		$$(\partial_t+\tilde{L}_{k_m})v_m(x,t) = \frac{1}{r_m^{\beta-1}\theta(r_m)}\left(f_{1,k_m} - p_{k_m,r_m}^{(0)}f_{2,k_m} - (\partial_t+\tilde{L}_{k_m})(p_{k_m,r_m}^{(1)}d_{k_m})\right)(r_mx,r_m^2t).$$
		We want to approximate the right-hand side as well as possible with a polynomial. By assumption $f_{i,k_m}\in C^{\beta-1}_p$ and hence we can approximate it up to order $|x|^{\beta-1} +|t|^{\frac{\beta-1}{2}}$, that is 
		$$|f_{i,m}(x,t) - P_{i,m}(x,t)|\leq C\left(|x|^{\beta-1} + |t|^{\frac{\beta-1}{2}}\right),$$ 
		for suitable polynomials $P_{i,m}$.
		Let us turn now to the term with the distance function
		\begin{align*}
			(\partial_t+\tilde{L}_{k_m})(p_{k_m,r_m}^{(1)}d_{k_m}) = &(\partial_t+\tilde{L}_{k_m})p_{k_m,r_m}^{(1)} d_{k_m} + p_{k_m,r_m}^{(1)}(\partial_t+\tilde{L}_{k_m}) d_{k_m} \\ & - 2 \nabla p_{k_m,r_m}^{(1)} A_{k_m} \nabla d_{k_m}
		\end{align*}
		Let us first consider the case $\beta>2$. Then $d_{k_m}\in C^{\beta}_p$, and hence $\nabla d_{k_m}\in C^{\beta-1}_p$, and $(\partial_t+\tilde{L}_{k_m})d_{k_m}\in C^{\beta-2}_p$. 
		Therefore $d_{k_m}$ and $\nabla d_{k_m}$ can be approximated up to order $|x|^{\beta-1}+|t|^{\frac{\beta-1}{2}}$, while the remaining term can only be approximated up to order $|x|^{\beta-2}+|t|^{\frac{\beta-2}{2}}$. But this term is multiplied with $p_{k_m,r_m}^{(1)}$, which can be estimated with $|x|+|t|$, and hence we get
		$$ \left|(\partial_t + \tilde{L}_{k_m})(p_{k_m,r_m}^{(1)}d_{k_m})(x,t) - P'_m(x,t)\right|\leq C \left(|x|^{\beta-1} + |t|^{\frac{\beta-1}{2}}\right).$$
		Thus we deduce
		$$ |(\partial_t+\tilde{L}_{k_m})v_m(x,t)-P_m(x,t)| \leq \frac{C ||p_{k_m,r_m}||}{\theta(r_m)}\left(|x|^{\beta-1} + |t|^{\frac{\beta-1}{2}}\right),$$
		for a suitable polynomial $P_m\in \textbf{P}_{\lfloor\beta-1\rfloor,p}.$ Note that $\frac{ ||p_{k_m,r_m}||}{\theta(r_m)}\to 0$, as $m\to\infty.$
		
		In the case $\beta\in(1,2)$, we have $d_{k_m}\in C^{\beta}$, and $|D^2d_{k_m}|\leq Cd^{\beta-2}.$ Hence we get 
		$$|(\partial_t+\tilde{L}_{k_m})v_m(x,t)-P_m(x,t)| \leq \frac{C ||p_{k_m,r_m}||}{\theta(r_m)}\left(|x|^{\beta-1} + |t|^{\frac{\beta-1}{2}} + (|x|+|t|)d_{k_m}^{\beta-2}\right).$$
		
		In both cases we get local uniform convergence of the residue $o_m:=(\partial_t+\tilde{L}_{k_m})v_m(x,t)-P_m(x,t)$ to zero. Denote $\Omega_m = \{(x,t);\text{ }(r_mx,r_m^2t)\in \Omega_{k_m} \} $ and split $v_m = v_{m,1}+v_{m,2}$, with
		$$\left\lbrace\begin{array}{rcl l}
			(\partial_t  + \tilde{L}_{k_m})v_{m,1} & =& P_m &\text{in } \Omega_m\cap Q_1\\
			v_{m,1}& =& 0 &\text{in } \partial_p(\Omega_{m}\cap Q_1), 
		\end{array} \right.$$ and 
		$$		\left\lbrace\begin{array}{rcl l}
		(\partial_t  + \tilde{L}_{k_m})v_{m,2} & =& o_m &\text{in } \Omega_{m}\cap Q_1\\
		v_{m,2}& =& v_m &\text{in } \partial_p(\Omega_{m}\cap Q_1).
		\end{array} \right.$$ 
		Note that the existence of $v_{m,1},v_{m,2}$ is given by \cite[Theorem 5.15]{Li96}.
		Since $v_m$ are bounded independently of $m$, we can get a bound $|v_{m,2}|\leq C d_{m},$ with $C$ independent of $m$, using a barrier constructed in Lemma \ref{barrier2}. Since $\Omega_m$ are closer and closer to $\{x_n>0\},$ we can estimate $|v_{m,2}|\leq C$, independently of $m$. Hence we have uniform bounds for $||v_{m,1}||_{L^\infty(Q_1)}$ as well, which by Lemma \ref{boundednesOfPolynomial} implies also the uniform boundedness of $||P_m||.$ Then we apply Corollary \ref{lipschitzBounds} rescaled to any $Q_R$, to get uniform Lipschitz bounds for $v_m$. Together with the convergence result \cite[Theorem 1.1]{Ba05} this ensures that a subsequence of the blow-up sequence converges to some function $v$ satisfying
		$$\left\lbrace\begin{array}{rcll}
		(\partial_t+L)v &=& P&\text{in }\{x_n>0\}\\
		v&=&0&\text{on }\{x_n=0\},\\
		||v||_{L^\infty(Q_R)}&\leq &CR^{\beta+1}&\text{for all } R\geq 1\\
		||v||_{L^\infty(Q_1)}&\geq& \frac{1}{2}&\\
		\int_{Q_1}v p(x_n)_+ &=& 0& \text{for every }p\in  \textbf{P}_{\lfloor\beta\rfloor,p},
		\end{array}\right.$$
		where $P\in \textbf{P}_{\lfloor\beta-1\rfloor,p}.$ Proposition \ref{Liouville} then says that $v$ equals $qx_n$, $q\in \textbf{P}_{\lfloor\beta\rfloor,p}$, which gives the contradiction with the last two properties of $v$, as wanted.
		
		Finally we prove also \eqref{interRegForExpansion}. If $c_\Omega$ is big enough, then for every point $(x_0,t_0)$ with $d(x_0,t_0) = c_\Omega r$ we have $Q_{2r}(x_0,t_0)\subset \Omega$. It suffices applying Proposition \ref{interiorEstimates} on  the rescaled function $v_r(x,t) = (u_1-p_{z,t_0}u_2)(x_0+2rx, t_0+4r^2t)$, take into account the just proven bound and noticing that $||p_{(z,t_0)}||$ can be uniformly bounded thanks to Lemma \ref{uniformBoundOnPolynomials}.
	\end{proof}

	We now have all ingredients to prove Theorem \ref{regularityOfQuotient}.

	\begin{proof}[Proof of Theorem \ref{regularityOfQuotient}]
		The proof goes along the same lines as Theorem \ref{regularityOfQuotient2}, just that the estimates with higher order parabolic seminorms are a bit more complicated.
		
		Note that 
		$$\left[ \frac{u_1}{u_2}\right] _{C^\beta_p(\Omega\cap Q_{1/2})} = \left[ D^{\lfloor\beta\rfloor}_p\frac{u_1}{u_2}\right] _{C^{\langle\beta\rangle}_p(\Omega\cap Q_{1/2})} + \left[ D^{\lfloor\beta\rfloor-1}_p\frac{u_1}{u_2}\right] _{C^{\frac{1+\beta}{2}}_t(\Omega\cap Q_{1/2})}.$$
		We start with estimating the first term.
		
		Choose $\gamma\in \N^{n+1}$, with $|\gamma|_p = \lfloor\beta\rfloor$ and a cylinder $Q_r(x_0,t_0)$, so that $Q_{2_r}(x_0,t_0)\subset \Omega\cap Q_{3/4},$ and $d(x_0,t_0)\leq Cr$, with $C$ independent of $x_0,t_0$ and $r$. Let $z$ be the closest point to $x_0$ in $\partial\Omega_{t_0}$.
		Take arbitrary $(x,t)$ and $(y,s)$ in $Q_r(x_0,t_0)$, denote $p = p_{(z,t_0)}\in\textbf{P}_{\lfloor\beta\rfloor,p}$ and compute
		$$\left|\partial^\gamma\frac{u_1}{u_2}(x,t) -\partial^\gamma\frac{u_1}{u_2}(y,s)\right| = 
		\left|\partial^\gamma\left(\frac{u_1}{u_2}(x,t) -p(x,t) \right) -\partial^\gamma\left(\frac{u_1}{u_2}(y,s)-p(y,s) \right)\right| \leq $$
		\begin{align}\label{termWithExpanisons}
			\begin{split}
			 \leq & \sum_{\alpha\leq\gamma} \left| \partial^\alpha\left( u_1 - p u_2 \right)(x,t) - \partial^\alpha\left( u_1 - p u_2 \right)(y,s) \right|\left| \partial^{\gamma-\alpha} u_2^{-1}(x,t) \right|+\\
			&+  \sum_{\alpha\leq\gamma} \left| \partial^\alpha\left( u_1 - p u_2 \right)(y,s) \right|\left| \partial^{\gamma-\alpha} u_2^{-1}(x,t)-\partial^{\gamma-\alpha} u_2^{-1}(y,s) \right|.
			\end{split}
		\end{align}
		By Proposition \ref{harnackExpansion} and Lemma \ref{fullRegularityFromCylinders} 
		we have that $u_1 - p u_2$ is $C^{\beta+1}_p(\overline{\mathcal{C}}),$ for a suitable cone $\mathcal{C} = \cup_{d_{t_0}(x_0) = c_\Omega r=|x_0-z|} Q_r(x_0,t_0)$, with $D^k_p(u_1 - p u_2)(z,t_0)=0$, for all $k< \beta+1$. Hence by Lemma \ref{polynomialAproximation} we can estimate 
		$$|D^k_p (u_1 - p u_2) (x,t)|\leq C (|x-z|^{\beta+1-k}+|t-t_0|^{\frac{\beta+1-k}{2}}), \quad (x,t)\in\mathcal{C}, k< \beta+1.$$
		
		Next take $\alpha \in \N^{n+1}$  and compute 
		$$\partial^\alpha u_2^{-1} = \sum_{l\leq |\alpha|} \frac{1}{u_2^{l+1}}\sum_{\alpha_1+\ldots+\alpha_l=\alpha}c_{\alpha_1,\ldots,\alpha_l}\partial^{\alpha_1}u_2\cdots \partial^{\alpha_l}u_2.$$
		Since by Corollary \ref{boundaryRegularity} $u_2\in C^\beta_p(\Omega\cap Q_1)$
		this implies 
		$$||\partial^\alpha u_2^{-1}||_{L^\infty(Q_r(x_0,t_0))}\leq C r^{-(|\alpha|+1)},$$
		when $|\alpha|_p<\beta,$ as well as 
		\begin{equation}\label{temporary1}
			\left[ \partial^\alpha u_2^{-1}\right]_{C^{\langle\beta\rangle}_p(Q_r(x_0,t_0))}\leq C r^{-(|\alpha|+1)},
		\end{equation}
		when $|\alpha|_p = \lfloor\beta\rfloor,$
		and 
		\begin{equation}\label{temporary2}
			\left[ \partial^\alpha u_2^{-1}\right]_{C^{\frac{\langle\beta\rangle+1}{2}}_t(Q_r(x_0,t_0))}\leq C r^{-(|\alpha|+1)},
		\end{equation}
		when $|\alpha|_p = \lfloor\beta\rfloor-1.$
		
		We are now equipped to estimate \eqref{termWithExpanisons}. Take $\alpha\leq \gamma$, with $|\alpha|_p<\beta$  and estimate
		\begin{align*}
			|\partial^\alpha v(x,t)-\partial^\alpha v(y,s)|\leq& |\partial^\alpha v(x,t)-\partial^\alpha v(y,t)|+|\partial^\alpha v(y,t)-\partial^\alpha v(y,s)|\leq \\
			\leq& ||D^{|\alpha|_p+1}_p v|| |x-y| + ||D^{|\alpha|_p+2}_p v|| |t-s|\\
			\leq & C r^{\beta+1-|\alpha|_p-1}r^{1-\langle\beta\rangle}|x-y|^{\langle\beta\rangle} + r^{\beta+1-|\alpha|_p-2}r^{2(1-\frac{\langle\beta\rangle}{2})}|t-s|^{\frac{\langle\beta\rangle}{2}}\\
			\leq & C r^{\lfloor\beta\rfloor +1-|\alpha|_p}(|x-y|^{\langle\alpha\rangle} + |t-s|^{\frac{\langle\beta\rangle}{2}}).
		\end{align*}
		If $|\alpha|_p = \lfloor\beta\rfloor$, then we need to use the definition of $\left[\partial^\alpha v \right] _{C^{\frac{1+\langle\beta\rangle}{2}}_t}$, to get the same estimate.
		Similarly, if $|\gamma-\alpha|_p<\beta-2$,
		\begin{align*}
		|\partial^{\gamma-\alpha}u_2^{-1}(x,t)-\partial^{\gamma-\alpha}u_2^{-1}(y,s)|\leq & ||D^{|\gamma-\alpha|+1} u_2^{-1}|| (|x-y|+|t-s|)\\
		\leq &  C r^{-|\gamma-\alpha|-2}(r^{1-\langle\beta\rangle}|x-y|^{\langle\beta\rangle}+ r^{2(1-\frac{\langle\beta\rangle}{2})}|t-s|^{\frac{\langle\beta\rangle}{s}})\\
		\leq &  C r^{-|\gamma-\alpha|-1-\langle\beta\rangle}(|x-y|^{\langle\beta\rangle}+|t-s|^{\frac{\langle\beta\rangle}{2}}).
		\end{align*}
		When $|\gamma-\alpha|_p = \lfloor\beta\rfloor-1$ we have to split the difference into the spatial part and time part and on the time part use \eqref{temporary2}. When $|\gamma-\alpha|_p = \lfloor\beta\rfloor$, we directly use \eqref{temporary1}
		to get the same estimate in other cases. Plugging it all in expression \eqref{termWithExpanisons}, we get
		$$\eqref{termWithExpanisons}\leq \sum_{\alpha\leq\gamma} Cr^{\lfloor\beta\rfloor +1-|\alpha|_p}Cr^{-|\gamma-\alpha|-1} (|x-y|^{\langle\beta\rangle}+|t-s|^{\frac{\langle\beta\rangle}{2}})+$$
		$$ + 
		\sum_{\alpha\leq\gamma} Cr^{\beta+1-|\alpha|_p}r^{-|\gamma-\alpha|-1-\langle\beta\rangle}(|x-y|^{\langle\beta\rangle}+|t-s|^{\frac{\langle\beta\rangle}{2}}).$$
		It remains to notice that $|\gamma-\alpha|\leq|\gamma-\alpha|_p$, and hence $\eqref{termWithExpanisons}\leq C (|x-y|^{\langle\beta\rangle}+|t-s|^{\frac{\langle\beta\rangle}{2}})$.
		
		We argue similarly that $\left[ D^{\lfloor\beta\rfloor-1}_p\frac{u_1}{u_2}\right] _{C^{\frac{1+\beta}{2}}_t(\Omega\cap Q_{1/2})}\leq C$. Choose $|\gamma|_p = \lfloor\beta\rfloor-1$ remember the degree of $p$ and compute 
		$$\left|\partial^\gamma\frac{u_1}{u_2}(x,t) -\partial^\gamma\frac{u_1}{u_2}(x,s)\right| = 
		\left|\partial^\gamma\left(\frac{u_1}{u_2}(x,t) -p(x,t) \right) -\partial^\gamma\left(\frac{u_1}{u_2}(x,s)-p(x,s) \right)\right| \leq $$
		\begin{align*}
		\begin{split}
		\leq & \sum_{\alpha\leq\gamma} \left| \partial^\alpha\left( u_1 - p u_2 \right)(x,t) - \partial^\alpha\left( u_1 - p u_2 \right)(x,s) \right|\left| \partial^{\gamma-\alpha} u_2^{-1}(x,t) \right|+\\
		&+  \sum_{\alpha\leq\gamma} \left| \partial^\alpha\left( u_1 - p u_2 \right)(x,s) \right|\left| \partial^{\gamma-\alpha} u_2^{-1}(x,t)-\partial^{\gamma-\alpha} u_2^{-1}(x,s) \right|\\
		\leq & \sum_{\alpha\leq\gamma} ||D^{|\alpha|_p+2}_p v|| \cdot|t-s|\cdot||D^{|\gamma-\alpha|}u_2^{-1}||+\\
		&+  \sum_{1\leq|\alpha|\leq|\gamma|} ||D^{|\alpha|_p}_p v||\cdot || D^{|\gamma-\alpha|+1} u_2^{-1}||\cdot |t-s|\\
		& + ||v||\cdot \left[\partial^\gamma u_2^{-1}\right]_{C^{\frac{1+\langle\beta\rangle}{2}}_t} \cdot|t-s|^{\frac{1+\langle\beta\rangle}{2}}\\
		\leq & \sum_{\alpha\leq\gamma} C r^{\beta+1-|\alpha|_p-2}r^{2(\frac{1-\langle\beta\rangle}{2})}|t-s|^{\frac{1+\langle\beta\rangle}{2}} Cr^{-|\gamma-\alpha|-1}\\
		&+ \sum_{1\leq|\alpha|\leq|\gamma|}Cr^{\beta+1-|\alpha|_p}Cr^{-|\gamma-\alpha|-2}r^{2(\frac{1-\langle\beta\rangle}{2})}|t-s|^{\frac{1+\langle\beta\rangle}{2}} \\
		&+ Cr^{\beta+1}Cr^{-|\gamma|-1}r^{2(\frac{1-\langle\beta\rangle}{2})}|t-s|^{\frac{1+\langle\beta\rangle}{2}}\\
		\leq & C|t-s|^{\frac{1+\langle\beta\rangle}{2}}.
		\end{split}
		\end{align*}
		Therefore the claim is proven, thanks to Lemma \ref{fullRegularityFromCylinders}. 
		
		The same steps prove the claim in case $\beta=1$, but we use Proposition \ref{HarnackExpansion0} instead of Proposition \ref{harnackExpansion}, and the $C^\beta_p$  estimate for $u_2$ up to the boundary is replaced by $C^{1-\varepsilon}_p$ from Corollary \ref{lipschitzBounds}.
	\end{proof}
	
	\section{Higher regularity of free boundary in obstacle problem}\label{higherRegularityOfFree...}
	
	Our parabolic higher order boundary Harnack inequality gives a simple way to prove that once the free boundary in the parabolic obstacle problems is $C^{1}_p$, it is in fact $C^\infty$. The reason is that the normal to the free boundary can be expressed with the quotients of partial derivatives of the solution to obstacle problem, which are as smooth as the boundary. This gives the bootstrap argument, and does not require any use of a hodograph transform as in \cite{KN77}.
	
	Let us start with preliminary results that we need in the proof of Corollary \ref{infiniteRegularityOfFreeBoundary}. 
	
	\begin{lemma}\label{preliminaryResults}
		Let $v\colon Q_1\to\R$ solve \eqref{theParabolicObstacleProblem} with $f\in C^\theta(B_1)$, for some $\theta>0$, with $f(0) = 1$. Assume that $(0,0)\in\partial\{v>0\}$ is a regular free boundary point.
		
		Then $v\in C^{1,1}_p(Q_1)$, up to rotation of coordinates the free boundary is $C^1_p$ in $Q_r$ for some $r>0$ in the sense of Definition \ref{assumptionsOnDomain}, $v_t\in C(Q_r)$ and $\partial_\nu v \geq c_1 d$, with $c_1>0$, where $\nu$ is the unit spatial normal vector of the free boundary at $(0,0)$.
	\end{lemma}

	\begin{proof}
		The  $C^{1,1}_p(Q_1)$ regularity of solutions of the obstacle problem follows from \cite[Theorem 1.6]{LM11}. The $C^1_p$ regularity of the free boundary is provided in \cite[Theorem 1.7]{LM15}. Denote the boundary defining map with $F$. 
		
		Let us now prove the continuity of the time derivative.	This is well known when $f\equiv 1$ (see \cite{Ca77}). For completeness we prove it next in our setting.
		Since inside $\Omega$ the solution $v$ is $C^{\theta+2}_p$ and $v\equiv 0$ in $\Omega^c$, we have to show that 
		$$\lim_{(x,t)\to (z,s)} \partial_t v(x,t) = 0,$$ 
		for every free boundary point in $Q_r$. Choose a point $(x,t)\in\Omega\cap Q_r$, such that $(x,\tau)\in\Omega^c$ for $\tau\leq s$. Let $(x,s),(x',F(x',t),t)$ be free boundary points. Therefore by \cite[Theorem 1.6]{LM11} we have $v(x,t) \leq C |x_n-F(x',t)|^2.$ But since $F\in C^1_p$, and $x_n = F(x',s)$, we have $v(x,t)\leq C |F(x't)-F(x',s)|^2\leq  C(|t-s|^{\frac{1}{2}}\omega(|t-s|^\frac{1}{2}))^2$, where $\omega$ is the modulus of continuity of $\nabla F$. This gives 
		$|u(x,t)-u(x,s)| = u(x,t)\leq |t-s| \eta(|t-s|)$, for some modulus of continuity $\eta$.
		
		Finally, let us show that $\partial_\nu v \geq c_1 d$ where $c_1>0$, $d$ is the regularised distance to the $\partial\Omega$ and $\nu$ is its spatial normal vector at $(0,0)$. Without loss of generality assume that $\nu = e_n$. Having the bound $||u||_{C^{1,1}_p(Q_1)}\leq C$, together with \cite[Theorem 1.7]{LM11} and Arzela-Ascoli theorem imply that $v$ converges to its blow up at any free boundary point $(x_0,t_0)\in \partial\Omega\cap Q_r$ uniformly (independently of the free boundary point) in $C^{0,1}_p$, namely
		$$\left|\left|\frac{1}{r^2}v(x_0+rx,t_0+r^2t) - \frac{f(x_0)}{2}(x\nu_{x_0})_+^2\right|\right|_{C^{0,1}_p(Q_1)}\longrightarrow 0$$
		as $r\to 0,$ uniformly in $(x_0,t_0)$.
		Hence reading it for the $n-$th derivative, this means that there exists a number $r_0>0$ so that for all $r<r_0$ we have
		$$\left|\left|\frac{1}{r}v_n(x_0+rx,t_0+r^2t) - f(x_0)(\nu_{x_0} e_n)(x\nu_{x_0})_+ \right|\right|_{L^\infty (Q_1)}\leq \frac{1}{4}.$$
		Restricting ourselves to a potentially smaller neighbourhood of $(0,0)$, so that $\nu_{x_0}e_n>\frac{3}{4}$, $|\nu_{x_0}-e_n|\leq \frac{1}{5},$ $|f(x_0)-1|>\frac{1}{4}$ with a triangle inequality we deduce that for $(x,t_0)\in Q_r((x_0,t_0))\cap\{(x-x_0)e_n>\frac{r}{2}\}$ we have
		$$\partial_n v(x,t_0) \geq f(x_0)(\nu_{x_0}e_n)((x-x_0)\nu_{x_0})-\frac{r}{4}\geq \frac{r}{4}.$$ 
		But since $d_x(x,t_0)\leq |x-x_0| = r,$ and the estimate is independent of $(x_0,t_0)$, we get that $\partial_v\geq \frac{1}{4}d_x\geq \frac{1}{4C}d$, since the two distances are comparable.
	\end{proof}
	
	We are now equipped enough to prove Corollary \ref{infiniteRegularityOfFreeBoundary}.

	\begin{proof}[Proof of Corollary \ref{infiniteRegularityOfFreeBoundary}] 
		Denote $\Omega = \{v>0\}.$ By Lemma \ref{preliminaryResults}, rescaling and rotation if necessary $\Omega$ is $C^1_p$ in $Q_1$, the solution $v$ is $C^{1,1}$ in space and $C^1$ in time. This implies that in $\Omega^c$ both $v_t:=\partial_t v=0$ and for $i=1,\ldots,n$ also $v_i:=\partial_i v =0$.  Without loss of generality assume that the normal vector to $\Omega_0$ at $0$ is $e_n$. 
		Hence all the partial derivatives of $v$ solve 
		$$\left\lbrace\begin{array}{rcll}
		(\partial_t-\Delta)w &=& \partial_e f&\text{in }\Omega\cap Q_1\\
		w&=&0&\text{on }\partial\Omega\cap Q_1. 
		\end{array}\right.$$
		Remember that $f$ is independent of time and by assumption $f\in C^\theta(B_1)$ and hence $\partial_ef\in C^{\theta-1}_p(Q_1)$.
		Moreover Lemma \ref{preliminaryResults} says that $v_n\geq c_1 d$ with $c_1>0$ in a neighbourhood of $0$.
		Hence we can apply Theorem \ref{regularityOfQuotient}, which gives that all quotients $v_i/v_n$ and $v_t/v_n$ are $C^{1-\varepsilon}_p(\overline{\Omega}\cap Q_{r_2}),$ with bounds on the norms. 
		
		Now notice that every component the normal vector $\nu(x,t)$ to the level set $\{v=t\}$, $t>0$ can be expressed as
		$$\nu^i(x,t) = \frac{\partial_i v}{|\nabla_{(x,t)} v|}(x,t) = \frac{\partial_i v/\partial_nv}{\left(\sum_{j=1}^{n-1}(\partial_jv/\partial_nv)^2 + 1 + (\partial_tv/\partial_nv)^2\right)^{1/2}}.$$ 
		Letting $t\downarrow0$, we get that the normal vector is $C^{1-\varepsilon}_p(\partial \Omega\cap Q_{r/2}).$ Hence the free boundary is $C^{2-\varepsilon}_p$.
		
		The same reasoning gives that if the boundary is $C^\beta_p$ it is $C^{\beta+1}_p$, as long as $\beta \leq \theta$ and hence the claim is proven.
	\end{proof}

	\appendix
	
\section{Interior regularity results}\label{interiorRegularityResults}
	In this section we state the interior regularity results used in the body of the paper. Even though the results are not new, we provide the proofs since the claims are adapted to our specific setting.
	For overview of the theory of parabolic second order equations, we refer to \cite{Kr96,Li96}.
	
	Our techniques rely on contradiction and blow up arguments. The contradiction at the end is usually provided by Liouville type results. We start with establishing such kind of result, saying that if a function solves an equation in the full space $\R^{n+1}$ with a polynomial in the right-hand side, then the function is a polynomial as well. The proof is based on using interior estimates iteratively on arbitrarily big cylinders.
	
	\begin{proposition}\label{LiouvilleFullSpace}
		Assume that $v$ solves 
		$$\left\lbrace\begin{array}{rcl l}
		\partial_t v - \operatorname{tr}(AD^2v) & =& P &\text{in } \R^{n+1}\\
		||v||_{L^\infty(Q_R)}& \leq& C_0R^\gamma &\forall R>1,\\ 
		\end{array} \right.$$
		for some constant, uniformly elliptic matrix $A$, $\gamma>0,$ $\gamma\not\in\N$, and some polynomial $P$ of parabolic order less than $\gamma-2$ (if $\gamma<2$, then $P=0$). Then $v$ is a polynomial of parabolic order less than $\gamma$.
	\end{proposition}  
	\begin{proof}
		Take $R>1$ and define 
		$$v_R(x,t) := R^{-\gamma} v(Rx,R^2t),$$
		so that $(\partial_t v_R - \operatorname{tr}(AD^2v_R))(x,t) =R^{-\gamma+2} P(Rx,R^2t) = P_R(x,t).$ Note that $P_R$ is again a polynomial of the same order and the coefficients reduce as $R$ increases. Moreover, by assumption $||v_R||_{L^\infty(Q_1)}\leq C_0$. 
		Interior regularity result \cite[Theorem 2]{PS19} gives
		$$\left[ v_R\right] _{C_p^{0,1}(Q_{1/2})}\leq C(||P_R|| + ||v_R||_{L^\infty(Q_1)})\leq C(||P||+C_0),$$
		and hence 
		$$\left[ v\right] _{C_p^{0,1}(Q_{R/2})}\leq  CR^{\gamma-1}(||P||+C_0).$$
		Choose arbitrary $h\in\R^n,$ $\tau\in\R$, and define
		$$w_1(x,t) = \frac{v(x+h,t+\tau)-v(x,t)}{|h|+|\tau|}.$$
		Then $w_1$ solves
		$$\partial_t w_1 - \operatorname{tr}(AD^2w_1)  = P_1 ,$$
		for some polynomial $P_1\in\textbf{P}_{\lfloor\gamma-1\rfloor,p}$
		and it satisfies the growth control
		$$||w_1||_{L^\infty(Q_R)}\leq C\left[ v\right] _{C^{0,1}(Q_{R})}\leq  CR^{\gamma-1}.$$
		
		Repeating the same procedure we conclude that 
		$$\left[ w_{\lceil\gamma\rceil}\right] _{C^{0,1}(Q_{R/2})}\leq  CR^{\gamma-\lceil\gamma\rceil},$$
		which gives that $w_{\lceil\gamma\rceil}$ is constant. This implies that $v$ is a polynomial, and the assumed growth assures that it is of parabolic order less than $\gamma$.
	\end{proof}
	
	Since we define the parabolic H\"older seminorms a bit differently than for example in \cite{Kr96}, we prove the a priori Schauder estimates in our setting as well.
	
	\begin{proposition}\label{interiorSchauder}
		Assume $u\in C^{2+\alpha}_p(Q_1)$ solves the equation
		$$\partial_tu + Lu = f,\quad\quad \text{in }Q_1,$$
		with $f\in C^\alpha_p(Q_1),$ $A\in C^\alpha_p(Q_1)$ and $b\in C^\alpha_p(Q_1)$. Then we have
		$$\left|\left| u\right|\right| _{C^{2+\alpha}_p(Q_{1/2})}\leq C(\left[ f\right] _{C^\alpha_p(Q_1)}+||u||_{L^\infty(Q_1)}),$$
		where $C$ depends only on $n,\alpha$, ellipticity constants and H\"older norms of the coefficients. 
	\end{proposition}	
	
	\begin{proof}
		Thanks to \cite[Lemma 2.23]{FR20}
		and interpolation inequality, it suffices to prove that for every $\delta>0$ there exists $C$ so that
		$$\left[ u\right] _{C^{2+\alpha}_p(Q_{1/2})}\leq \delta \left[ u\right] _{C^{2+\alpha}_p(Q_{1})} +  C(\left[ f\right]_{C^\alpha_p(Q_1)}+||u||_{C^2_p(Q_1)}).$$
		Dividing the equation with a constant, we can assume that $\left[ f\right]_{C^\alpha_p(Q_1)}+||u||_{C^2_p(Q_1)} = 1$.
		
		We argue with contradiction. Assume that there is $\delta>0$ so that for every $k\in\N$ there exist $u_k,f_k$ and $L_k$ $(\lambda,\Lambda)$-uniformly elliptic operators so that $\partial_t u_k + L_ku_k = f_k$ in $Q_1$, but
		$$\left[ u_k\right] _{C^{2+\alpha}_p(Q_{1/2})} > \delta \left[ u_k\right] _{C^{2+\alpha}_p(Q_{1})} + k.$$
		Choose points $(x_k,t_k),$ and $(y_k,s_k)$, so that 
		\begin{align*}
		c(n)\left[ u_k\right] _{C^{2+\alpha}_p(Q_{1/2})} \leq& \frac{|D^2u_k(x_k,t_k)-D^2u_k(y_k,s_k)|}{|x_k-y_k|^\alpha + |t_k-s_k|^{\frac{\alpha}{2}}} + \frac{|\partial_t u_k(x_k,t_k)-\partial_t u_k(y_k,s_k)|}{|x_k-y_k|^\alpha + |t_k-s_k|^{\frac{\alpha}{2}}}\\
		& + \frac{|Du_k(x_k,t_k)-Du_k(x_k,s_k)|}{ |t_k-s_k|^{\frac{1+\alpha}{2}}}.
		\end{align*}
		Define $\rho_k = |x_k-y_k| + |t_k-s_k|^{\frac{1}{2}}.$
		Then the above inequality implies
		\begin{align*}
		c(n)\left[ u_k\right] _{C^{2+\alpha}_p(Q_{1/2})} \leq& \frac{|D^2u_k(x_k,t_k)-D^2u_k(y_k,s_k)|}{\rho_k^\alpha } + \frac{|\partial_t u_k(x_k,t_k)-\partial_t u_k(y_k,s_k)|}{\rho_k^\alpha}\\
		& + \frac{|Du_k(x_k,t_k)-Du_k(x_k,s_k)|}{\rho_k^{1+\alpha}}.
		\end{align*}
		Applying triangle inequality and bounding the terms with $||u_k||_{C^2_p(Q_1)}$, we get
		$$c(n)\left[ u_k\right] _{C^{2+\alpha}_p(Q_{1/2})} \leq ||u_k||_{C^2_p(Q_1)}(\rho_k^{-\alpha} + \rho_k^{-1+\alpha}),$$
		but thanks to the contradiction assumption, 
		$$c(n)\left[ u_k\right] _{C^{2+\alpha}_p(Q_{1/2})} \leq \frac{\left[ u_k\right] _{C^{2+\alpha}_p(Q_{1/2})}}{k}(\rho_k^{-\alpha} + \rho_k^{-1+\alpha}),$$
		which assures that $\rho_k\to 0$ as $k\to \infty.$
		
		Now we define the blow-up sequence
		$$v_k(x,t) = \frac{1}{\rho_k^{2+\alpha}\left[ u_k\right] _{C^{2+\alpha}_p(Q_{1})}} \left(u_k(x_k+\rho_kx, t_k+\rho_k^2 t ) - p_k(x,t)\right),$$
		where $p_k$ is a polynomial of parabolic order $2$, so that $v_k(0,0) = Dv_k(0,0) = D^2v_k(0,0) = \partial_tv_k(0,0)=0$. Now take any $R<\frac{1}{2\rho_k}.$ Since $p_k$ is of lower order, it holds 
		$$\left[ v_k\right] _{C^{2+\alpha}_p(Q_{R})} = \frac{1}{\left[ u_k\right] _{C^{2+\alpha}_p(Q_{1})}} \left[ u_k\right] _{C^{2+\alpha}_p(Q_{R\rho_k}(x_k,t_k))}\leq 1,$$ 
		and hence also 
		$||v_k(x,t)||_{L^\infty(Q_R)}\leq R^{2+\alpha}$ and in particular $\left[ v_k\right] _{C^{2+\alpha}_p(Q_{1})}\leq 1$. Moreover, by definition of $(x_k,t_k)$ and $\rho_k$ we have
		$$\left[ v_k\right] _{C^{2+\alpha}_p(Q_{1})} \geq \delta c(n).$$ Furthermore
		\begin{align*}
		(\partial_t +\tilde{L}_k)v_k(x,t) = & \frac{1}{\rho_k^{\alpha}\left[ u_k\right] _{C^{2+\alpha}_p(Q_{1})}}   \Bigg (  f_k(x_k+\rho_kx,t_k+\rho_k^2t)-\\
		&-  \partial_tu_k(x_k,t_k) - \sum_{i,j}a^k_{i,j}(x_k+\rho_kx,t_k+\rho_k^2t)\partial_{ij}u_k(x_k,t_k)+\\
		& +\sum_i b^k_i(x_k+\rho_kx,t_k+\rho_k^2t) \partial_i u_k(x_k,t_k)+\\
		& + \rho_k\sum_{i} b^k_i(x_k+\rho_kx,t_k+\rho_k^2t) \sum_j \partial_{ij}u_k(x_k,t_k) x_j \Bigg ).
		\end{align*}
		Taking into account the regularity of the coefficients, with adding and subtracting suitable terms, we get
		\begin{align*}
		(\partial_t +\tilde{L}_k)v_k(x,t) = & \frac{1}{\rho_k^{\alpha}\left[ u_k\right] _{C^{2+\alpha}_p(Q_{1})}}   \Bigg ( f_k(x_k+\rho_kx,t_k+\rho_k^2t)-f_k(x_k,t_k)+\\
		&+ \sum_{i,j}o_{A,k}(x,t)\partial_{ij}u_k(x_k,t_k) + \sum_{i}o_{b,k}(x,t)\partial_{i}u_k(x_k,t_k) +\\
		&+\rho_k \sum_i b_i^k(x_k+\rho_k x,t_k+\rho_k^2t).\Bigg )
		\end{align*}
		and hence by contradiction assumption
		$$|(\partial_t +\tilde{L}_k)v_k(x,t)|\leq \frac{C}{k}\left( \left[ f_k\right]_{C^\alpha_p}+ ||A_k||_{C^\alpha_p}+||b_k||_{C^\alpha_p} + \rho_k^{1-\alpha}\Lambda\right),$$
		on any compact set inside $Q_{\frac{1}{2\rho_k}}$.
		
		Passing to a subsequence, Arzela-Ascoli theorem provides that $v_k$ converge in $C^{2}_p(\R^{n+1})$ on compact sets to some function $v$, which by \cite[Theorem 1.1]{Ba05} solves
		$$\left\lbrace \begin{array}{rcll}
		(\partial_t - \operatorname{tr}(AD^2))v &=& 0 & \text{in }\R^{n+1}\\
		||v||_{L^\infty(Q_R)}&\leq& R^{2+\alpha} &\text{for all }R>0\\
		v(0)=Dv(0)=D^2v(0)=\partial_t v (0)&=&0&\\
		\left[ v\right] _{C^{2+\alpha}_p(Q_{1})} &\geq& \delta c(n),&
		\end{array}  \right.$$
		for some uniformly elliptic, constant matrix $A$.
		Hence by Liouville theorem (Proposition \ref{LiouvilleFullSpace}) $v$  is a polynomial of parabolic order $2$, which is a contradiction with last two properties of $v$.
	\end{proof}
	
	Next we establish the higher order a priori estimates.
	
	\begin{corollary}
		Let $k\in\N$, $k\geq 2$, and $\alpha\in(0,1)$. Assume $u\in C^{k,\alpha}_p(Q_1)$ solves
		$$(\partial_t+L)u = f \quad\quad\text{in }Q_1,$$
		for some $f\in C^{k-2,\alpha}_p(Q_1)$, $A,b\in C^{k-2,\alpha}_p(Q_1)$. Then 
		$$\left|\left| u\right|\right| _{C^{k+\alpha}_p(Q_{1/2})}\leq C(\left[ f\right]_{C^{k-2,\alpha}_p(Q_1)}+||u||_{L^\infty(Q_1)}),$$
		where $C$ depends only on $n,\alpha$, ellipticity constants and H\"older norms of the coefficients. 
	\end{corollary}
	
	\begin{proof}
		We prove the claim by induction. Proposition \ref{interiorSchauder} proves the case $k=2$. So let $k>2$ and assume the claim is true for all $l<k$. Choose now $e\in \S^{n-1}.$ Then by induction hypothesis 
		$$\left[ \partial_e u\right] _{C^{k-1+\alpha}_p(Q_{1/2})}\leq C(\left[ \partial_ef\right] _{C^{k-3,\alpha}_p(Q_1)}+||\partial_eu||_{L^\infty(Q_1)}),$$
		which after interpolation inequality implies
		$$\left[ D u\right] _{C^{k-1+\alpha}_p(Q_{1/2})}\leq C(\left[ f\right]_{C^{k-2,\alpha}_p(Q_1)}+||u||_{L^\infty(Q_1)}).$$
		Hence also 
		$$\left[ D^2 u\right] _{C^{k-2+\alpha}_p(Q_{1/2})}\leq C(\left[ f\right]_{C^{k-2,\alpha}_p(Q_1)}+||u||_{L^\infty(Q_1)}),$$
		and so using that $\partial_t u = f -Lu$ we also get
		$$\left[ \partial_t u\right] _{C^{k-2+\alpha}_p(Q_{1/2})}\leq C(\left[f\right]_{C^{k-2,\alpha}_p(Q_1)}+||u||_{L^\infty(Q_1)}).$$
		Combined, the inequalities render
		$$\left[  u\right] _{C^{k+\alpha}_p(Q_{1/2})}\leq C(\left[ f\right]_{C^{k-2,\alpha}_p(Q_1)}+||u||_{L^\infty(Q_1)}),$$
		and the claim is proven.
	\end{proof}

	We conclude the section by showing that solutions to equations with regular right-hand side are indeed as regular as the a priori estimates indicate. 
	
	\begin{proposition}\label{interiorEstimates}
		Let $k\in\N$, $k\geq 2$, and $\alpha\in(0,1)$. Let $A,b,f\in C^{k-2,\alpha}_p(Q_1)$, and $u$ a solution to 
		$$(\partial_t+L)u = f,\quad \text{in }Q_1.$$
		Then $u\in C^{k,\alpha}_p(Q_1),$ with 
		$$\left| \left| u\right|\right| _{C^{k+\alpha}_p(Q_{1/2})}\leq C(\left[ f\right]_{C^{k-2,\alpha}_p(Q_1)}+||u||_{L^\infty(Q_1)}),$$
		where $C$ depends only on $n,\alpha,k$, ellipticity constants and H\"older norms of the coefficients. 
	\end{proposition}
	\begin{proof}
		We only need to prove the regularity of $u$, since we have already established the estimates.
		
		If $k=2$, then the regularity is justified by \cite[Theorem 5.9]{Li96}. For $k>2$, with the same argument as in \cite[Theorem 8.12.1]{Kr96} we establish that $\nabla u\in C^{k-1,\alpha}_p(Q_1)$. Using the equation, this also gives that $\partial_t u \in C^{k-2,\alpha}_p(Q_1)$.
	\end{proof}

	\section{Technical tools and lemmas}\label{technicalToolsAndLemmas}
	
	\begin{lemma}\label{polynomialAproximation}
		Let $\Omega\subset \R^{n+1}$ be a domain, such that  $\Omega_0$ is convex and $\{t;(t,x)\in\Omega\}$ is connected for every $x$. Let $\beta\not\in\N$ and assume $u\in C^\beta_p(\overline{\Omega})$, with $D^{k}_p u(0,0)=0$, for every $k\leq \lfloor\beta\rfloor$. Then for every $k\leq\beta$ and $r>0$ we have 
		$$||D^k_pu||_{L^\infty(\Omega\cap Q_r)}\leq Cr^{\beta-k}.$$
	\end{lemma}
	\begin{proof}
		We prove the claim with induction on $\beta$ and $k$.		If $k = \lfloor\beta\rfloor$, then it follows from the definition of $C^\beta_p$.
		
		If $k = \lfloor\beta\rfloor-1,$ we estimate
		$$|D^k_pu(x,t)|\leq |D^k_pu(x,t)-D^k_pu(x,0)|+ |D^ku(x,0)|\leq C|t|^{\frac{1+\langle\beta\rangle}{2}} + \int_0^{|x|} ||D^{k+1}_pu||_{L^\infty(Q_s)}ds,$$
		and hence by induction hypothesis
		$$|D^k_pu(x,t)|\leq C r^{1+\langle\beta\rangle}.$$
		
		If $k < \lfloor\beta\rfloor-1,$ we use $\partial_t D^k_pu \subset D^{k+2}_pu$ and estimate 
		$$|D^k_pu(x,t)-D^k_pu(x,0)|\leq \int_0^t ||D^{k+2}_pu||_{L^\infty(Q_s)}ds,$$
		to obtain the same result.
	\end{proof}

	\begin{lemma}\label{fullRegularityFromCylinders}
		Let $\beta> 0$ and let $\Omega\subset\R^{n+1}$ be $C^{\max(\beta,1)}_p$ in $Q_1$ in the sense of Definition \ref{assumptionsOnDomain}. Assume that a function $u\colon\Omega\cap Q_1\to\R$ satisfies 
		$$\left[ u\right] _{C^\beta_p(Q_{r}(x_0,t_0))}\leq C_0,$$
		whenever $Q_{2r}(x_0,t_0)\subset \Omega\cap Q_1$ and $d_{t_0}(x_0)\leq C_1 r$, with $C_1$ depending only on $\Omega $.
		
		Then
		$$\left[ u\right] _{C^\beta_p(\Omega\cap Q_{1/2})}\leq CC_0.$$
		The constant $C$ depends only on $n,\beta$ and $\Omega$.
	\end{lemma}

	\begin{proof}
		Thanks to parabolic scaling and covering argument, we can assume $C_1=3$, and that $\left[ u\right] _{C^\beta_p(Q_{\frac{3r}{2}}(x_0,t_0))}\leq C_0$, or in other words, we have 
		$$\left[ u\right] _{C^\beta_p(Q_{r}(x_0,t_0))}\leq C_0, $$
		whenever $d_{t_0}(x_0) = 2r.$
		
		Let us first prove, that  $\left[ u\right] _{C^\beta_p(\mathcal{C}_{(z,t_0)})}\leq CC_0,$ where
		$$\mathcal{C}_{(z,t_0)} = \cup_{d_{t_0}(x_0) = 2r=|x_0-z|} Q_{r}(x_0,t_0), $$ 
		is a "parabolic cone" starting at a boundary point $(z,t_0)\in\partial\Omega.$ First, notice that 
		$$\left[D^{\lfloor\beta\rfloor-1}_p u\right] _{C^{\frac{1+\langle\beta\rangle}{2}}_t(\mathcal{C}_{(z,t_0)}}\leq C_0,$$
		since the points we choose for comparing the function values have to lie in the same cylinder $Q_r(x_0,t_0)$, because the space coordinate coincide.
		
		Now choose arbitrary $(x,t)$ and $(y,s)$ in $\mathcal{C}_{(z,t_0)}$.  Without loss of generality assume $d(x,t)\leq d(y,s)$. Denote $x'$ and $y'$ the central points of cylinders from $\mathcal{C}_{(z,t_0)}$, containing $(x,t)$ and $(y,s)$ respectively. We want to estimate
		$$|D^{\lfloor\beta\rfloor}_p u(x,t)-D^{\lfloor\beta\rfloor}_p u(y,s)|\leq CC_0(|x-y|^{\langle\beta\rangle} + |t-s|^{\frac{\langle\beta\rangle}{2}}).$$
		If $2d_{t_0}(x)>d_{t_0}(y)$, both points are contained in $Q_{\frac{3}{4}d_{t_0}(y)}(y,t_0)$, and hence another covering argument assures the above estimate. Otherwise we have $d_{t_0}(y)\leq 2|x-y|$ as well as $|x-x'|\leq \frac{d_{t_0}(x)}{2}$ and $|y-y'|\leq \frac{d_{t_0}(y)}{2}.$ For $i=0,1,\ldots$ denote $x_i = x'+2^{-i}(y'-x')$ and $r_i = d_{t_0}(x_i)$, so that $(x_{i+1},t),(x_{i},t)\in Q_{\frac{r_i}{2}}(x_i,t_0)$. Hence we have
		\begin{align*}
			|D^{\lfloor\beta\rfloor}_p u(x,t)-D^{\lfloor\beta\rfloor}_p u(y,s)|\leq & |D^{\lfloor\beta\rfloor}_p u(y,s)-D^{\lfloor\beta\rfloor}_p u(y',t)|+\\& + \sum_{i=0}^K |D^{\lfloor\beta\rfloor}_p u(x_i,t)-D^{\lfloor\beta\rfloor}_p u(x_{i+1},t)|+\\
			& + |D^{\lfloor\beta\rfloor}_p u(x,t)-D^{\lfloor\beta\rfloor}_p u(x',t)|\\
			\leq & C_0(|y-y'|^{\langle\beta\rangle} + |t-s|^{\frac{\langle\beta\rangle}{2}}) + \sum_{i=0}^K C_0|x_i-x_{i+1}|^{\langle\beta\rangle} \\
			&+ C_0|x-x'|^{\langle\beta\rangle} \\
			\leq &  C_0\left( 2|x-y|^{\langle\beta\rangle} + |t-s|^{\frac{\langle\beta\rangle}{2}} + |x-y|^{\langle\beta\rangle} \sum_{i=0}^\infty 2^{-i\langle\beta\rangle} \right) \\
			=& C_\beta C_0\left( |x-y|^{\langle\beta\rangle} + |t-s|^{\frac{\langle\beta\rangle}{2}}\right).
		\end{align*}
		
		Let us now turn to the general case. Pick $(x,t),(y,s)\in\Omega\cap Q_{1/2}$. Denote $z_x\in\partial \Omega_t$, $z_y\in\partial\Omega_s$ so that $d_t(x) = |x-z_x|$ and $d_s(y)=|y-z_y|$. Now choose $x_1\in \Omega_t$, a central point of the cylinder from $\mathcal{C}_{(z_x,t)}$, so that $|x-x_1|\leq 2|t-s|^{\frac{1}{2}}$, and that $(x_1,s)\in \mathcal{C}_{(z_x,t)}$. If the boundary is flat enough, this is possible, meaning that we have to restrict ourselves to a small neighbourhood of the boundary before if necessary. Now take $z_{xy}\in\partial\Omega_s$, so that $d_s(x_1) = |x_1-z_{x,y}|,$ and choose a point $y_1\in\Omega_s$, so that $(y_1,s)\in \mathcal{C}_{(z_{xy},s)}\cap \mathcal{C}_{(z_y,s)},$ but $|x_1-y_1|\leq 2 \left(|x-y| + |t-s|^{\frac{1}{2}}\right)$. Note that $(x,t)\in \mathcal{C}_{(z_x,t)},$ as well as $(y,s)\in \mathcal{C}_{(z_y,s)},$ and hence
		\begin{align*}
			|D^{\lfloor\beta\rfloor}_p u(x,t)-D^{\lfloor\beta\rfloor}_p u(y,s)|\leq& |D^{\lfloor\beta\rfloor}_p u(x,t)-D^{\lfloor\beta\rfloor}_p u(x_1,s)|+|D^{\lfloor\beta\rfloor}_p u(x_1,s)-D^{\lfloor\beta\rfloor}_p u(y_1,s)|\\
			&+|D^{\lfloor\beta\rfloor}_p u(y_1,s)-D^{\lfloor\beta\rfloor}_p u(y,s)|\\
			\leq & C_\beta C_0 \left(|x-x_1|^{\langle\beta\rangle} + |t-s|^{\frac{\langle\beta\rangle}{2}}\right) + C_\beta C_0 |x_1-y_1|^{\langle\beta\rangle}  \\
			&+ C_\beta C_0 |y_1-y|^{\langle\beta\rangle}\\
			\leq & CC_0 \left(|x-y|^{\langle\beta\rangle} + |t-s|^{\frac{\langle\beta\rangle}{2}}\right).
		\end{align*}
		
		Finally, choose two points $(x,t),(x,s)\in\Omega\cap Q_{1/2}$. Let $d_t(x) = |x-z_t|$ and $d_s(x)=|x-z_s|$. Pick $x_1$ and $x_2$ in $\Omega_t\cap\Omega_s$, so that $x,x_1,x_2$ lie on the same line, $|x-x_1 = |x_1-x_2|$, $|x-x_1|\leq 2|t-s|^{\frac{1}{2}}$, and that $(x,t),(x_1,t),(x_2,t),(x_1,s),(x_2,s)\in \mathcal{C}_{(z_t,t)}$. If $\Omega$ is flat enough, this is possible, otherwise we restrict ourselves to the smaller neighbourhood of the boundary. Therefore we have
		\begin{align*}
			|D^{\lfloor\beta\rfloor-1}_p u(x,t)-D^{\lfloor\beta\rfloor-1}_p u(x,s)|\leq & |D^{\lfloor\beta\rfloor-1}_p u(x,t)-2D^{\lfloor\beta\rfloor-1}_p u(x_1,t) + D^{\lfloor\beta\rfloor-1}_p u(x_2,t)| \\
			&+ |D^{\lfloor\beta\rfloor-1}_p u(x,s)-2D^{\lfloor\beta\rfloor-1}_p u(x_1,s) + D^{\lfloor\beta\rfloor-1}_p u(x_2,s)| \\
			&+ 2|D^{\lfloor\beta\rfloor-1}_p u(x_1,t)-D^{\lfloor\beta\rfloor-1}_p u(x_1,s)|\\
			&+ |D^{\lfloor\beta\rfloor-1}_p u(x_2,t)-D^{\lfloor\beta\rfloor-1}_p u(x_2,s)|\\
			\leq & 2 CC_0|x-x_1|^{1+\langle\beta\rangle}  + 3CC_0|t-s|^{\frac{1+\langle\beta\rangle}{2}}\\
			\leq & CC_0 |t-s|^{\frac{1+\langle\beta\rangle}{2}},
		\end{align*}
		and so the claim is proven.
	\end{proof}
	\begin{remark}
		We need the boundary to be at least $C^{0,1}_p$, so that the cones $\mathcal{C}_{(z,s)}$ come up to the boundary.
	\end{remark}

	\begin{lemma}\label{copied4.3}
		Let $\beta >0$ $\beta\not\in\N$, $\Omega\subset\R^{n+1}$, with $0\in\partial\Omega$, and let $u\in C(\overline{ Q}_1).$ Assume that for every $r\in(0,1)$ we have a polynomial $p_r\in \textbf{P}_{\lfloor\beta\rfloor,p}$ so that 
		$$||u-p_rd||_{L^\infty(Q_r)}\leq C_0r^{\beta+1}.$$
		Then there exists a polynomial $p_0\in \textbf{P}_{\lfloor\beta\rfloor,p}$ which satisfies
		$$||u-p_0d||_{L^\infty(Q_r)}\leq CC_0r^{\beta+1},\quad r\in(0,1),$$
		where $C$ depends only on $n,$ and $\beta$.
	\end{lemma}
	\begin{proof}
		The proof is the same as the one of \cite[Lemma 4.3]{AR20}, just that every $B_r$ is replaced with $Q_r$, and polynomials are of bounded parabolic degree.
	\end{proof}
	
	\begin{lemma}\label{uniformBoundOnPolynomials}
		Let $\Omega\subset\R^{n+1}$, and assume that for every point $z\in\partial\Omega\cap Q_1$ there exists a polynomial $p_z\in \textbf{P}_{K}$, so that 
		$$||u(x)-p_z(x)v(x)||_{L^\infty(B_1(z)\cap \Omega)}\leq C_0,$$
		for some functions $u,v\in L^\infty(\Omega)$, with $C_0$ independent of $z$. Assume also that $v\geq d$. 
		
		Then we have $||p_z||\leq C C_0,$ for some $C$ independent of $z$.
	\end{lemma}
	
	\begin{proof}
		Find a non-empty, open set $B\subset \cap_{z\in\partial\Omega\cap B_1} B_1(z),$ that is away from the boundary $B\subset \{x;\text{ }d(x)>r \}$, for some $r>0$. Then we can estimate
		\begin{align*}
			||p_z||_{L^\infty(B)}\leq \frac{1}{r}||p_z v||_{L^\infty(B_1(z))}&\leq \frac{1}{r}\left(||u-p_z v||_{L^\infty(B_1(z))} + ||u||_{L^\infty(B_1(z))}\right)\\&\leq C(C_0 + ||u||_{L^\infty(\Omega)}).
		\end{align*}
		The claim follows from \cite[Lemma A.10]{AR20}.
	\end{proof}

	\begin{lemma}\label{boundednesOfPolynomial}
		Let $D\subset\R^{n+1}$ be a bounded domain and let $v$ be a solution to 
		$$\left\lbrace \begin{array}{rcll}
		(\partial_t + L)v &=& P & \text{in }D\\
		v&=&0& \text{in }\partial_p D,
		\end{array}\right.$$
		For some polynomial $P$. Then 
		$$||P||_{L^\infty(D)}\leq C ||v||_{L^\infty(D)},$$
		where $C$ depends only on $n,D,$ ellipticity constants and the degree of $P$.
	\end{lemma}
	
	\begin{proof}
		If $P=0$ we have nothing to prove. Otherwise, dividing the equation with a constant we can assume that $||P||_{L^\infty(D)}=1$. Assume by contradiction, that there are sequences of $(\lambda,\Lambda)$-uniformly elliptic operators $L_k$, polynomials $P_k$ of fixed degree, with $||P_k||_{L^\infty(D)}=1$, and solutions $v_k$ of 
		$$\left\lbrace \begin{array}{rcll}
		(\partial_t + L_k)v_k &=& P_k& \text{in }D\\
		v_k&=&0& \text{in }\partial_p D,
		\end{array}\right.$$
		and that 
		$$\frac{1}{k}>  ||v_k||_{L^\infty(D)}.$$ 
		Since $P_k$ are bounded in a finite dimensional vector space, we can get a subsequence converging uniformly in $D$ to some polynomial $P_0$. Hence by convergence result \cite[Theorem 1.1]{Ba05}
		the limit function $v_0$ should satisfy
		$$\left\lbrace \begin{array}{rcll}
		(\partial_t + L_0)v_0 &=& P_0& \text{in }D\\
		v_0&=&0& \text{in }\partial_p D.
		\end{array}\right.$$
		But $v_0\equiv 0$, while $||P_0||_{L^\infty(D)}=1$, which gives a contradiction.
	\end{proof}
	
	\begin{lemma}\label{copied4.5}
		Let $\beta >0$ $\beta\not\in\N$, $\Omega\subset\R^{n+1}$, with $0\in\partial\Omega$, and let $u_1,u_2\in C(\overline{ Q}_1).$ Assume that for every $r\in(0,1)$ we have a polynomial $p_r\in \textbf{P}_{\lfloor\beta\rfloor,p}$ so that 
		$$||u_1- p_r^{(0)}u_2 - p_r^{(1)}d||_{L^\infty(Q_r)}\leq C_0r^{\beta+1}.$$
		Furthermore assume $c_0d\leq u_2\leq C_0d$, for some $c_0>0$.
		Then there exists a polynomial $p_0\in \textbf{P}_{\lfloor\beta\rfloor,p}$ which satisfies
		$$||u_1- p_0^{(0)}u_2 - p_0^{(1)}d||_{L^\infty(Q_r)}\leq CC_0r^{\beta+1},\quad r\in(0,1),$$
		where $C$ depends only on $n,c_0$ and $\beta$.
	\end{lemma}
	
	\begin{proof}
		The proof is the same as the one of \cite[Lemma 4.5]{AR20}, with $Q_r$ instead of $B_r$.
	\end{proof}

\end{document}